\documentclass[a4paper,12pt,reqno]{amsart}
\usepackage{amssymb}
\usepackage{amsfonts,amsmath,amstext,amsbsy,amsopn,amsthm,amscd}

\usepackage[]{mathtools}
\mathtoolsset{showonlyrefs}
\usepackage[noadjust]{cite}

\usepackage[shortlabels]{enumitem}

\newcommand{\RRe}{\mathop{\rm Re}}

\newcommand{\eps}{\varepsilon}

\newcommand{\ft}{\mathfrak t}
\newcommand{\x}{\mathrm x}
\newcommand{\y}{\mathrm y}
\newcommand{\z}{\mathrm z}
\newcommand{\embeds}{\hookrightarrow}

\newcommand{\lowdim}[1]{\mathbf{#1}}

\newtheorem{theorem}{Theorem}[section]
\newtheorem{lemma}[theorem]{Lemma}%[section]
\newtheorem{corollary}[theorem]{Corollary}%[section]
\newtheorem{proposition}[theorem]{Proposition}%[section]
\theoremstyle{definition}

\newtheorem{assu}[theorem]{Assumption}%[section]
\newtheorem{rem}[theorem]{Remark}%[section]

\theoremstyle{remark}

\numberwithin{equation}{section}

\newcommand{\field}[1]{\mathbb{#1}}

\newcommand{\R}{\field{R}}

\DeclareMathOperator{\dom}{dom}    %domain of an operator
\DeclareMathOperator{\dist}{dist}  %distance
\DeclareMathOperator{\supp}{supp}  %support of a function
\DeclareMathOperator{\diag}{diag}   %divergence

\newcommand{\VV}{V} 
 
\newcommand{\dd}{\,\mathrm{d}} 
\newcommand{\cA}{\mathcal A}

\newcommand{\cH}{\mathcal{H}}

\newcommand{\divv}{\operatorname{div}}

\usepackage{hyperref}
\begin{document}
\title[H\"older regularity for domains of fractional powers]{H\"older regularity for domains of
  fractional powers of 
  elliptic operators with 
  mixed boundary conditions}

\author{Robert Haller, Hannes Meinlschmidt, Joachim Rehberg}

\begin{abstract}
  This work is about global H\"older regularity for solutions to elliptic partial differential
  equations subject to mixed boundary conditions on irregular domains. There are two main
  results. In the first, we show that if the domain of the realization of an elliptic
  differential operator in a negative Sobolev space with integrability $q>d$ embeds into a space
  of H\"older continuous functions, then so do the domains of suitable fractional powers of this
  operator. The second main result then establishes that the premise of the first is indeed
  satisfied. The proof goes along the classical techniques of localization, transformation and
  reflection which allows to fall back to the classical results of Ladyzhenskaya or
  Kinderlehrer. One of the main features of our approach is that we do not require Lipschitz
  charts for the Dirichlet boundary part, but only an intriguing metric/measure-theoretic
  condition on the interface of Dirichlet- and Neumann boundary parts. A similar condition was
  posed in a related work by ter~Elst and Rehberg in 2015~\cite{ER}, but the present proof is
  much simpler, if only restricted to space dimension up to $4$.
\end{abstract}

\maketitle

\section{Introduction} \label{s-introduc}
In this paper, we consider global H\"older regularity for solutions to elliptic partial differential equations
subject to mixed boundary conditions on irregular domains, in the exemplary form
\begin{equation}
  \left.
    \begin{aligned}%[t]
    -\divv(\mu\nabla u) + u & = f &&\text{in}~\Omega,\\u & = 0 &&\text{on}~D \subseteq
    \partial\Omega,\\ \nabla u \cdot \nu & = g && \text{on}~N\coloneqq \partial\Omega\setminus D
  \end{aligned}
  \quad \right\}\label{eq:model-problem}
\end{equation}
for a bounded open set $\Omega\subseteq \R^d$ with the unit outer normal $\nu$ at $N$, a bounded
and elliptic
coefficient function $\mu$ taking its values in $\R^{d\times d}$, and integrable functions $f$
on $\Omega$ and $g$ on $N$. It is well known that H\"older continuity is a natural regularity
class for solutions to elliptic problems such as~\eqref{eq:model-problem} and
H\"older-equicontinuous sets of functions are precompact in the space of uniformly
continuous functions by the Arzel\`{a}-Ascoli theorem. Such properties are, aside from intrinsic
value, invaluable in the
treatment of nonlinear problems. It is thus not surprising that this is a
well researched subject and affirmative results are known even in the case of irregular domains
and mixed boundary conditions with very weak compatibility conditions as established for example
in~\cite{ER} by one of the authors.

The intention of this paper is essentially twofold: Firstly, we prove that if the domain
$\dom(\cA_q+1)$ of the functional-analytic realization $\cA+1$ of the elliptic differential
operator in~\eqref{eq:model-problem} in a negative Sobolev space $W^{-1,q}_D(\Omega)$ embeds
into a space of H\"older-continuous functions, then so does the domain $\dom((\cA_q+1)^\sigma)$
of a fractional power of $\cA+1$ when $\sigma > \frac12 + \frac{d}{2q}$. (We will introduce all
objects properly in the main text below.) It is well known that $q>d$ is the expected condition
in this context. This is done under the quite general assumption that $\overline N$ admits
bi-Lipschitzian boundary charts and $D$ is Ahlfors regular; the coefficient function $\mu$ is
not supposed to be more than measurable, bounded and elliptic. (See Assumption~\ref{a-general} below.) The main
motivation for this result are semilinear parabolic problems, since it is well known that since
the semigroup associated to $\cA_q+1$ will be analytic, the domain $\dom((\cA_q+1)^\sigma)$ will
be a natural phase space, see e.g.~\cite[Ch.~6.3]{Pazy}. We will come back to this below in a
bit more detail.

Secondly, we consider a framework where the assumption of the first part is in fact satisfied;
that is, we show that $\dom(\cA_q+1)$ indeed embeds into a H\"older space. This framework will
essentially encapsulate the geometric assumptions from the first part, together with a classical
assumption preventing outward cusps for $D$, and an intriguing metric/measure-theoretic
condition for the interface of $D$ and $N$, the Dirichlet- and Neumann boundary parts, which
will ultimately allow to show that also at this interface, we can transform the problem under
consideration to one which satisfies the foregoing classical assumption. (See
Assumption~\ref{a-Interface} below.) To this
end, we revisit~\cite{ER} where the associated result was already established by means of
Sobolev-Campanato spaces and of De~Giorgi estimates.  These  are both quite natural and powerful, but also
quite involved. However, for spatial dimensions $d$ up to $4$ one can avoid this machinery and
rely on the classical results on H\"older continuity for solutions of the pure Dirichlet problem
by Ladyzhenskaya and Kinderlehrer, which require much simpler technical means. We carry out this
simplified approach here. A welcome byproduct  is that we in fact easily obtain a result which is
uniform in the given geometry and the $L^\infty(\Omega)$-bound and ellipticity constant of the
coefficient function $\mu$. Such statements are extremely useful in the treatment of, say, a
quasilinear counterpart of~\eqref{eq:model-problem}, and they are neither included in~\cite{ER} nor easily
traced there.

\subsection*{Motivation} It was already mentioned above that one of the main motivations to consider H\"older regularity
for $\dom(\cA_q + 1)$ and associated domains of fractional powers comes from semilinear
parabolic equations. Indeed, consider the following abstract one, posed in some Banach space $X$:
\begin{equation} \label{e-semiparab}
u'(t) + \cA u(t) + u(t) = F(t,u(t)), \quad u(0)=u_0,
\end{equation}
where $\cA+1$ is the realization of an elliptic operator such as the one
in~\eqref{eq:model-problem} in $X$.  The way to treat such a problem by means of analytic
semigroups is well established by now under weak assumptions on $F$, which require that the
coordinate mappings $t \mapsto F(t,v)$ for fixed $v$ and $v \mapsto F(t,v)$ for fixed $t$ are
reasonably well behaved, cf.~\cite[Ch.~6.3]{Pazy}, the latter usually referring to Lipschitz
continuity on bounded sets of the domain of a fractional power of $\cA+1$. A most interesting and
relevant case is that of Nemytskii operators induced by scalar functions; these for example
occur naturally in the form of polynomials in reaction-diffusion problems. Whether the
abstract framework can capture these nonlinearities depends on the precise framework and associated growth
properties and is usually the central point to verify when doing analysis for such problems. In fact, in the most prominent case $X = L^2(\Omega)$ and space dimensions up to $3$, one can
show that not only the domain of the elliptic operator $\cA+1$ in $L^2(\Omega)$ embeds into
$L^\infty(\Omega)$, but already the domain of a fractional power does so. This is established in
an even more general context than the present one in~\cite{ERe1}, but see also~\cite[Chapter
6.1]{Ouh} and Corollary~\ref{c-2} below. Since bounded functions are, essentially, ignorant of
growth induced by a Nemytskii operator, this allows to consider very rough nonlinearities $F$
induced by such operators.

However, this strong property comes at a price, namely that a realization of $\cA+1$ in $L^2(\Omega)$ implicitly
restricts the considered problem to a strong interpretation with homogeneous Neumann boundary
conditions. But this setup is in general insufficient for more sophisticated problems
arising in real world applications. This already concerns nonhomogeneous Neumann
boundary data. But also, consider for example a (two-dimensional) surface $S$
in the (closure of the) domain $\Omega \subset \R^3$ and let $\cH_2|_S$ be the induced two-dimensional
surface measure. Let $\phi$ be a scalar and locally Lipschitz function and let $\Phi$ be
the associated Nemytskii operator.  Suppose that $F$ in~\eqref{e-semiparab} is given by $v
\mapsto \Phi(v)  \cH_2|_S$.  Such a term would correspond to a nonlinear modulation for a
jump-type condition for the solution $u(t)$ along $S$ in a strong problem formulation, and, indeed, such
conditions appear for example in the analysis of the semiconductor equations if surface charge densities,
concentrated on $S$,  are involved, see~\cite{Hannes,Disser} for a recent analytical treatment; see
also~\cite{Semicond1,Semicond2} for more physical background. (In this particular example, there
are also nonlinear modulations on the boundary.)

Clearly, in such a setup, it is not sufficient to have
$\dom((\cA+1)^\sigma) \embeds L^\infty(\Omega)$ only, since this will in general not be enough
to interpret, much less control, $\Phi(v)$ on the lower-dimensional surface $S$ in dependence of
$v \in \dom((\cA+1)^\sigma)$. Alternatively, one could try to rely on trace operators to have a
good control on $v \in L^r(S;\cH_2)$ and then $\Phi(v)$ for $r$ large enough in dependence on
the growth conditions of $\phi$. But this in turn would require to pass through a Sobolev space
$W^{s,p}_D(\Omega)$ with $s>1/p$ and justifying such a setup might be quite hard if one goes
away from $(s,p) = (1,2)$, whereas the latter is rather limited, at least for $d=3$.

From our point of view, it is thus preferable to rely on H\"older continuity for the domain of a
fractional power of $\cA+1$. Then elements from such a domain are well defined on any subset of
$\overline\Omega$ and, as mentioned above, there are even compactness properties to exploit. It
turns out that the negative Sobolev space $W^{-1,q}_D(\Omega)$, which is the (anti-) dual of
$W^{1,q'}_D(\Omega)$, with $q>d$, provides the adequate functional-analytic framework $X$ to
obtain this H\"older continuity for the domain of a fractional power of the $X$-realization of
$\cA+1$, and then treat problems such as~\eqref{e-semiparab} with inhomogeneous data on
lower-dimensional surfaces in $\overline\Omega$, be that $\partial\Omega$ or $S$. Indeed,
negative Sobolev spaces are capable of representing distributional objects such as induced by
inhomogenenous data on lower-dimensional surfaces, and as already mentioned above, it is well
known that $q>d$ is the natural threshold for which one can obtain bounded or even continuous
functions as elements of the domain of the associated realization $\cA_q+1$, that is, for
solutions $u$ to the abstract problem $(\cA_q+1)u = f$ with $f \in W^{-1,q}_D(\Omega)$.

\subsection*{Context} As explained above, H\"older regularity for elliptic problems such
as~\eqref{eq:model-problem} is a classical and ubiquitous subject in the regularity theory for
partial differential equations. We locate our work between~\cite{ER} with essentially the same,
extremely general geometric setup, but a much more sophisticated and involved machinery to
achieve the desired result (without a direct claim of uniformity), and~\cite{HaR}, where the
less general framework of Gr\"oger regularity is used. The technique of the present work, in
terms of localization of an elliptic problem~\eqref{eq:model-problem} and associated
transformation to regular sets plus a possible reflection argument, is similar to the one
employed in~\cite{HaR}, but deviates from there along the different assumptions on $D$. We note
also that while there is no uniformity statement in~\cite{HaR}, there is the recent
preprint~\cite{Dondl} in which the authors there trace the constants in~\cite{HaR} to obtain a
uniform results, which is then even transferred to solutions of parabolic problems. In all
mentioned works, the coefficient function is also only assumed to be measurable, bounded and elliptic, as in the
present one.

\subsection*{Overview} We set the stage with notation and the introduction of function spaces
and differential operators with some associated properties in
Section~\ref{ss-prelim}. Section~\ref{sec-Embed} then deals with the first main result,
Theorem~\ref{t-main1}: if the domain of the $W^{-1,q}_D(\Omega)$-realization of $\cA+1$ embeds
into a H\"older space, then so does the domain of a fractional power. The proof is based on
ultracontractivity of the semigroups associated to the $L^p(\Omega)$-realization of $\cA+1$,
which we transfer to the negative Sobolev scale via the Kato square root
property. Section~\ref{s-Hoelderdomai} then deals with showing that the premise of the foregoing
part is in fact satisfied in a wide geometric setting in Theorem~\ref{t-result}. For this
result, the proof is somewhat extensive. We thus prepare it with a series of preliminary results
on the techniques of localization, transformation and reflection in
Section~\ref{sec:local-transf-techn} before proceeding to the actual meat of the proof in
Section~\ref{sec:proof-theorem-reft}.
\section{Preliminaries} \label{ss-prelim}
We first clarify some basic notation. The spatial dimension will be $d > 1$. For $\x = (x_1,\dots,x_d) \in \R^d$ and $r > 0$ we denote the open
ball around $\x$ with radius $r$ by $B_r(\x)$. The $d$-dimensional Lebesgue
measure in $\R^d$ will be written as $\lambda_d$ and $\omega_d = \lambda_d(B_1(0))$ means the
volume of the unit ball. Given
a normed vector space $\VV$, we denote by $\VV^*$ the Banach space of \emph{antilinear} continuous
functionals on $\VV$. Finally, we use the convention of a generic constant $c$ that may vary
from occurence to occurence but never depends on the free variables in the actual context. All
other notation will be standard.

\subsection{Function spaces}

Let $\Lambda$ be a nonempty, bounded open subset of $\R^d$ and let $F \subseteq \partial\Lambda$
be a closed subset of its boundary. Then, for
$q \in [1,\infty]$, the first-order Sobolev space $W^{1,q}(\Lambda)$ is given by the set of
$L^q(\Lambda)$ functions with weak first-order derivatives in $L^q(\Lambda)$. We set
\[
C^\infty_F(\Lambda)\coloneqq  \Bigl\{ u|_\Lambda \colon u \in C^\infty_c(\R^d) \text{ with } \mathrm{supp}(u) \cap F = \emptyset \Bigr\}
\]
and  we use this space to define the first-order \emph{Sobolev space with
  mixed boundary conditions} $W^{1,q}_F(\Lambda)$ as the closure of $C^\infty_F(\Lambda)$ in
$W^{1,q}(\Lambda)$. Furthermore, by $W^{-1,q}_F(\Lambda) \coloneqq W^{1,q'}_F(\Lambda)^*$ we denote
the space of continuous \emph{antilinear} functionals on $W^{1,q'}_F(\Lambda)$, where (here and
in all what follows) $1/q + 1/q' = 1$. Finally, as commonly used we write $W^{1,q}_0(\Lambda)$
for $W^{1,q}_{\partial \Lambda}(\Lambda)$ and $W^{-1,q}(\Lambda)$ for $W^{-1,q}_{\partial
  \Lambda}(\Lambda)$.

For $\alpha \in (0,1)$, let $C^\alpha(\Lambda)$ denote the usual spaces of bounded and
$\alpha$-H\"older continuous functions on $\Lambda$ with their norm given by the sum of the supremum
norm and the H\"older seminorm. Of course, every function in $C^\alpha(\Lambda)$ admits a unique
$\alpha$-H\"older continuous
extension to $\overline\Lambda$, so we will not discriminate between a H\"older-function on
$\Lambda$ and $\overline\Lambda$.

\subsection{Geometric setup}
We next introduce some geometric assumptions on the spatial domain $\Omega$. Throughout the
article, $\Omega$ denotes a given nonempty bounded open subset of $\R^d$ and
$D \subseteq \partial \Omega$ is a closed portion of its boundary, the designated Dirichlet
boundary part. We do not exclude that $\cH_{d-1}(D) = 0$, the $(d-1)$-dimensional
Hausdorff measure. The Neumann boundary part shall be denoted by
$N \coloneqq \partial \Omega \setminus D$.
\begin{assu} \label{a-general} We consider the following geometric assumptions for $\Omega$ and $D$:
\begin{enumerate}[(a)]
\item \label{a-general-1}
 For all $\x \in \overline{N}$, there is an open neighbourhood $V_\x$ and a bi-Lipschitz mapping
 $\phi_\x$ from a neighbourhood of $\overline{ V_\x}$ into $\R^d$ such that $\phi_\x(V_\x) =
 (-1,1)^d$, $\phi_\x( \Omega \cap  V_\x) = \{ \x \in (-1,1)^d \colon x_d < 0 \}$,
 $\phi_\x(\partial \Omega \cap V_\x) = \{ \x \in (-1,1)^d \colon x_d = 0 \}$ and $\phi_\x(\x) =
 0$. 
\item \label{a-general-2}
  $D$ is a $(d-1)$-set, i.e., there are constants $c_1, c_2 > 0$ such that
  for all $r \in (0,1]$ and all $\x \in D$ there holds
 \[ c_1 r^{d-1} \le \cH_{d-1} \bigl( B_r(\x) \cap D \bigr) \le c_2 r^{d-1} % \qquad
   % \x \in D,~r \in (0,1],
 \]
 where $\cH_{d-1}$ denotes the $(d-1)$-dimensional Hausdorff measure.
%\item \label{a-general-3}The set $D$ is so fat that the constant functions do not belong to $W^{1,2}_D(\Omega)$.
\end{enumerate}
\end{assu}
\begin{rem} \label{r-shrink} % Some comments on Assumption~\ref{a-general}:
 % \begin{enumerate}[(a)]
 % \item\label{r-shrink:i}
   In Assumption~\ref{a-general}~\ref{a-general-1}, for
   $\x \in N = \partial \Omega \setminus D$ one may assume without loss of generality that the
   local Neumann boundary part around $\x$ is transformed to the full midplate of the cube, that is,
   $\phi_\x (N \cap V_\x) = \{ \x \in (-1,1)^d \colon x_d = 0\}$. In fact, since $N$ is a
   (relatively) open subset of $\partial \Omega$, the image $\phi_\x(N \cap V_\x)$ is a
   (relatively) open subset of $\{ \x \in (-1,1)^d \colon x_d = 0\}$ that contains $0$. Thus, one may
   shrink $V_\x$ to a suitable set $\phi_\x^{-1}((-\eps,\eps)^d)$ and afterwards rescale
   $\phi_\x$ to $\frac{1}{\eps} \phi_\x$.
 %  \item\label{r-shrink:ii} Note that even when $D \neq \emptyset$, Assumption~\ref{a-general}~\ref{a-general-2} does not imply part~\ref{a-general-3} of the same assumption.\footnote{Haben wir da ein sch\"ones Beispiel/ein Zitat daf\"ur?} However, this is true under some additional mild condition, cf.~\cite[Lemma~7.3]{hardy}, e.g.\@ it suffices that $D$ contains one (relatively) inner point.
 % \end{enumerate}
\end{rem}
Already the geometric setup of Assumption~\ref{a-general}~\ref{a-general-1} allows to construct
a continuous linear extension operator for first-order Sobolev spaces with mixed boundary
conditions. Indeed, the following result can be found in \cite[Thm.~1.2 and Prop.~3.4]{BBHT}:
\begin{proposition} \label{p-2} Suppose that $\Omega$ and $D$ meet
  Assumption~\ref{a-general}~\ref{a-general-1}.  Then there exists a continuous extension
  operator from $W^{1,1}_D(\Omega)$ to $W^{1,1}_D(\R^d)$ that restricts to a continuous operator
  from $W^{1,p}_D(\Omega)$ to $W^{1,p}_D(\R^d)$ for all $p \in [1,\infty)$.
\end{proposition}
\begin{rem}\label{r-p2}
  Proposition~\ref{p-2} allows to establish the usual Sobolev embeddings, that is,
  $W^{1,q}_D(\Omega) \embeds L^p(\Omega)$ for $\frac1p = \frac1q-\frac1d$ if $q < d$ and
  $W^{1,q}_D(\Omega) \embeds C^{1-\frac{d}q}(\Omega)$ if $q > d$, in a straightforward manner,
  including compactness. In particular, for $d > 2$ the form domain $\VV = W^{1,2}_D(\Omega)$
  is embedded into $L^\frac {2d}{d-2}(\Omega)$, and in the case $d=2$ it embeeds into $L^p(\Omega)$ for every
  $p < \infty$.
\end{rem}

\subsection{Elliptic operators} We define elliptic operators via the form $\ft$ on $\VV
\coloneqq  W^{1,2}_D(\Omega)$ given by
\[
\ft(u,v) \coloneqq  \int_\Omega \mu \nabla u \cdot \nabla \overline v , \qquad u,v \in \VV .
\]
Here, $\mu$ is a real, measurable, bounded and uniformly elliptic coefficient function in the
sense that there exists some $\kappa_{\mathrm{ell}} > 0$ such that
$(\mu(\x) \xi, \xi)_{\R^d} \ge \kappa_{\mathrm{ell}} |\xi|^2$ for all $\xi \in \R^d$ and almost
all $\x \in \Omega$. Clearly, the form $\ft$ induces a natural operator
$\cA \colon \VV \to \VV^*$. For $q>2$, let $\cA_q$ be the part of $\cA = \cA_2$ in $W^{-1,q}_D(\Omega) \subset \VV^*$. By the Lax-Milgram lemma, $\cA + \lambda$ is a topological isomorphism between $V$
and $V^*$ for every $\lambda$ with $\RRe \lambda > 0$; hence,
$\sigma(\cA_q) \cap [\RRe z < 0] = \emptyset$ for every $q \geq2 $.

On the other hand, $\ft$ also induces an operator $A$ on $L^2(\Omega)$ by
\begin{align*}
  \dom A &\coloneqq \Bigl\{u \in V\colon \text{there exists } f \in L^2(\Omega) \colon \ft(u,v) =
  (f,v)_{L^2(\Omega)} \text{ for all } v \in V\Bigr\} \\
  Au & \coloneqq f, \text{ for } u \in \dom A.
\end{align*}
Since $\ft$ is $L^2(\Omega)$-elliptic, it is nowadays classical (e.g.~\cite[Thms.~1.54,~4.2
and~4.9]{Ouh}) that $-A$ is the generator of a contractive analytic C$_0$-semigroup $(e^{-At})$ on
$L^2(\Omega)$ which is both sub-Markovian and substochastic, that is, positivity preserving and
$L^\infty(\Omega)$- and $L^1(\Omega)$-contractive, from which we obtain the semigroup on every
$L^p(\Omega)$ for $p \in [1,\infty]$ by interpolation.

These semigroups are contractive for all $p \in [1,\infty]$ , they are strongly continuous for
$p \in [1,\infty)$, and they are analytic for $p \in (1,\infty)$, see~\cite[Prop.~3.12,
p.56/57\&96]{Ouh}. We denote the respective (negative) generators on $L^p(\Omega)$ by
$A_p$. Note that $\sigma(A_p) \cap [\RRe z < 0] = \emptyset$ for every $p\in [1,\infty)$ by the
Hille-Yosida theorem, and that the operators admit a bounded $H^\infty$ functional
calculus~(\cite[Cor.~3.9]{CMR}); in particular, their fractional powers are well
defined. Moreover, for $p>2$, the operators $A_p$ are the part of $A = A_2$ in $L^p(\Omega)$.

% We then consider the following operators
% for $q >2$:
% \begin{itemize}
%  \item $\mathcal A:\VV \to \VV^*$ is the operator induced by $\ft$,
%  \item $\mathcal A_q$ denotes the part of $\mathcal A$ in $W^{-1,q}_D(\Omega)$,
%  \item $A$ is the operator induced by $\ft$ on $L^2(\Omega)$.
%  \item $A_q$ denotes the part of $A$ in $L^q(\Omega)$.
% \end{itemize}

% The operators $A_q$ enjoy the following properties:
% %
% \begin{proposition} \label{p-LpPart}
%  Let $q>2$. Then $A_q$ is the part of $\mathcal A_q$ in  $L^q(\Omega)$ and does generate a contractive, analytic semigroup on $L^q(\Omega)$. The spectrum of $A_q$ does not meet the open negative half axis.\footnote{Beweis? Verweis?}
% \end{proposition}
%
All the properties mentioned so far do not require \emph{any} regularity assumption on
$\Omega$. Under the geometric assumptions from Assumption~\ref{a-general}, however, we can say a
bit more. Indeed, for $q\geq2$, several of the good properties of $A_q$ can be
transferred to $\cA_q$ by means of the square root, which we do next. 

\begin{proposition} \label{p-propAq} Let $q \in [2, \infty)$ and adopt
  Assumption~\ref{a-general}. Then the following hold
  true.
\begin{enumerate}[(a)]
\item \label{p-propAq2} The inverse square root operator $(\cA_q +1)^{-1/2}$ 
  provides a topological isomorphism between $W^{-1,q}_D(\Omega)$ and $L^q(\Omega)$.
\item \label{p-propAq1}
 The negative of the operator $\cA_q$ generates an analytic semigroup on $W^{-1,q}_D(\Omega)$.
\item \label{p-propAq3} For $s \in [0, \frac12)$, we have $\dom\bigl ((\cA_q + 1)^{1/2 + s} \bigr) = \dom\bigl( (A_q+1)^s \bigr)$.
 \end{enumerate}
\end{proposition}
\begin{proof}
  In~\cite[Thm.~1.1]{bechtel} it is proved that $A+1$ has the Kato square root property in the
  present geometric setting. (And even beyond that.) Using this fundamental property, the
  claim~\ref{p-propAq2} is one of the main results in~\cite{Auscher}, see Theorem~5.1
  there. Further, since $(\cA_q +1)^{-1}$ and $(A_q +1)^{-1}$ coincide on $L^q(\Omega)$, so do
  the inverse square roots, and we
  have the similarity 
 \[
  (\cA_q +\lambda )^{-1} = (\cA_q +1 )^{1/2} ( A_q +\lambda )^{-1} (\cA_q +1)^{-1/2}.
\]
Hence, we can transfer the  generator
property for an analytic semigroup from $-A_q$ to $-\cA_q$ by means of resolvent estimates, see the
characterization in~\cite[Thm.~II.4.6]{EngelNagel}. (Note that we do not claim the semigroups
generated by $-\cA_q$ to be contractive.) This implies~\ref{p-propAq1}. Finally, the fractional
powers of $\cA_q$ are well defined since the bounded $H^\infty$ calculus also transfers from
$A_q$ to $\cA_q$ by means of the square root~(\cite[Thm.~11.5]{Auscher}). Then,~\ref{p-propAq3}
follows immediately from~\ref{p-propAq2} by sketching
\begin{equation*}
  \bigl(\cA_q + 1\bigr)^{-1/2-s}W^{-1,q}_D(\Omega) = \bigl(\cA_q + 1\bigr)^{-s}L^q(\Omega) = \bigl(A_q + 1\bigr)^{-s}L^q(\Omega).\qedhere
\end{equation*}
\end{proof}
%%%%%%%%%%%%%%%%%%%%%%%%%%%%%%%%%%%%%%%%%%%%%%%%%%%%%%%%%
\section{Embeddings for domains of fractional powers of $\cA_q+1$}\label{sec-Embed}

In this section we show that if the domain of $\cA_q+1$ embeds into a H\"older space, so do
suitable fractional powers of this operator. We remark on the domain of $\cA_q$ after the proof
of Theorem~\ref{t-main1}. The question of \emph{when} the domain of $\cA_q+1$ actually embeds
into a H\"older space will be considered in Section~\ref{s-Hoelderdomai}.
%\subsection{The semigroup approach} \label{ss-semigroup}
%%%%%%%%%%%%%%%%%%%%%%%%%%%%%%%%%%%%%%%%%%%%%%%%%%%%%%%%%%
%We aim to show the following result.
%
\begin{theorem} \label{t-main1}
Let Assumption~\ref{a-general} be satisfied and let
$q>d$. Suppose that $\dom(\cA_q+1) \embeds C^\alpha(\Omega)$ for some $\alpha >
0$. Let $\kappa \in (0,\alpha)$ and $\sigma \in \bigl(\frac12+\frac{d}{2q} + \frac\kappa\alpha(\frac12 - \frac{d}{2q}), 1\bigr)$. Then we have 
\[ \bigl( W^{-1,q}_D(\Omega), \dom(\cA_q+1) \bigr)_{\sigma, 1} \hookrightarrow
  C^{\kappa}(\Omega)  \qquad \text{and} \qquad
 \dom\bigl( (\cA_q + 1)^\sigma \bigr) \embeds C^{\kappa}(\Omega).
\]
\end{theorem}
%
% Our proof of Theorem~\ref{t-main1} bases on the following abstract result on ultracontractivity for semigroups that is proved in~\cite[Chapter~7.3]{Are6}:
% %
% \begin{theorem} \label{l-02}
%  Let $H^1_0(\Omega) \subseteq \VV \subseteq H^1(\Omega)$ be the domain of a closed,
%  $L^2(\Omega)$-elliptic\footnote{Hannes: Hier stand mal 'sectorial'--wieso? (Wir haben das auch
%    nicht eingefuehrt..) Ich habe zudem $H^1_0 \subseteq V$ dazugeschrieben, damit man auch
%    sicher $V\cap L^1$ dicht in $L^1$ hat, wie es bei Arendt gefordert ist.} form and let $B$
%  be the operator induced by this form on $L^2(\Omega)$. Suppose that the semigroup $t \mapsto e^{-tB}$ generated by $-B$ in $L^2(\Omega)$ consistently extends to a bounded $C_0$-semigroup on $L^1(\Omega)$ and that the semigroup which is induced on $L^\infty$ is a bounded one. Finally, assume that $\VV$ embeds into $L^\frac{2\gamma }{\gamma -2}(\Omega)$ for some real number $\gamma > 2$.
%  Then for every $p \in (1,\infty )$
%  \begin{equation} \label{e-semigroupdecay2}
%   \|e^{-tB}  \|_{\mathcal L(L^p;L^\infty)} \le C t^{- \frac {\gamma}{2p}}, \qquad t > 0.
% \end{equation}
% \end{theorem}
% Note that under the assumptions of this theorem, the semigroup $e^{tB}$ extends to a bounded
% $C_0$-semigroup on $L^p(\Omega)$ for all $p \in [1,\infty)$. We will denote its generator by
% $B_p$.

Before we start with the proof, a short remark:

\begin{rem}
  \label{rem:heat-kernels}
  Via Proposition~\ref{p-propAq}, we also obtain from Theorem~\ref{t-main1} that
  \begin{equation*}
    \dom((A_q+1)^\varsigma) \embeds C^\kappa(\Omega)
  \end{equation*}
  for
  $\varsigma \in \bigl(\frac{d}{2q} + \frac\kappa\alpha(\frac12 - \frac{d}{2q}),
  \frac12\bigr)$. This is interesting because there is natural connection between embeddings of
  the domain of a fractional power of $A_q+1$ into a H\"older space and the H\"older continuity
  of the heat kernel associated to the semigroup generated by the negative of $A_q + 1$. We
  refer to~\cite[Ch.~6.2]{Ouh} and leave the details to the interested reader.
\end{rem}

Our proof of Theorem~\ref{t-main1} is based on ultracontractivity of semigroups generated
by $-A_q$. We use ultracontractivity to derive a precise regularizing property for inverse
fractional powers of $A_q+1$ and then in turn transfer this to the $\cA_q$ operator by means of
Proposition~\ref{p-2}.

The semigroups $(e^{-A_pt})$ are said to be \emph{ultracontractive} if there exists a constant $c>0$ and some $\gamma> 2$ such that
\begin{equation}\label{e-semigroupdecay2}
  \bigl\|e^{-A_p t}\bigr\|_{L^p(\Omega)\to L^\infty(\Omega)} \leq c t^{-\frac\gamma{2p}}
  \quad \text{for all } t >0,~p\in [1,\infty).
\end{equation}
In fact, this property is equivalent to $\VV \embeds L^{\frac{2\gamma}{\gamma-2}}(\Omega)$; we
refer to~\cite[Chapter~7.3]{Are6}. But under the geometric assumptions of
Assumption~\ref{a-general}~\ref{a-general-1}, Proposition~\ref{p-2} provides a Sobolev extension
operator from which the foregoing Sobolev embedding for $\VV$ with $\gamma = d$ if $d>2$ and any
$\gamma \in (2,\infty)$ if $d=2$ follows immediately as noted in Remark~\ref{r-p2}. This is
already the proof of the next proposition:

\begin{proposition}[Ultracontractivity]
  \label{p-ultracontract}
  Adopt Assumption~\ref{a-general}~\ref{a-general-1}. Then the semigroups $(e^{-A_pt})$ are
  \emph{ultracontractive}, that is, there exists $c>0$ such that~\eqref{e-semigroupdecay2} holds
  true for $\gamma = d$ if $d>2$ and $\gamma >2$ arbitrary if $d=2$.
\end{proposition}

We infer the following regularizing property for the inverse fractional powers of $A_p+1$
for $p>d/2$ :

\begin{corollary} \label{c-2} Adopt Assumption~\ref{a-general}~\ref{a-general-1} and let
  $p > d/2$. Then, for every $\tau \in (\frac {d}{2p}, 1]$, we find
  $(A_p+1)^{-\tau} \in \mathcal L(L^p(\Omega) \to L^\infty(\Omega))$. In particular,
  $\dom((A_p+1)^\tau) \embeds L^\infty(\Omega)$.
% \[ \|u \|_{L^\infty(\Omega)} \le c\,\bigl\|(A_p+1)^\tau u \bigr\|_{L^p(\Omega)} = c \|u\|_{\dom((A_p +1)^\tau)}, \qquad u \in \dom ( (A_p +1)^\tau ).
% \]
\end{corollary}

% \begin{corollary} \label{c-1}
% In the situation of Theorem~\ref{l-02}, let $p \in (\frac \gamma 2, \infty)$.\footnote{Hier stand $(1, \infty)$, aber das passt nicht zur Wahl von $\tau$.} Then for every $\tau \in (\frac {\gamma}{2p}, 1]$ one has $(B+1)^{-\tau} \in \mathcal L(L^p;L^\infty)$ and for all $u \in \dom ( (B_p +1)^\tau )$
% \[ \|u \|_{L^\infty} \le c \|(B+1)^\tau u \|_{L^p} = c \|u\|_{\dom((B_p +1)^\tau)}.
% \]
% \end{corollary}
\begin{proof}
Consider the well-known Balakrishnan formula
\[
(A_p+1)^{-\tau} = \frac {1}{\Gamma (\tau)} \int_0^\infty t^{\tau -1} e^{-A_pt}e^{-t} \,\mathrm{d}t.
\]
From Proposition~\ref{p-ultracontract} and the growth bound~\eqref{e-semigroupdecay2} for
$(e^{-A_pt})$, one observes immediately that the condition $\tau > \frac {d}{2p}$ is sufficient to
have the integral converge in $\mathcal L(L^p(\Omega)\to L^\infty(\Omega))$ and the claim follows. (For $d=2$, squeeze
$\frac\gamma{2p}$ between $\frac{d}{2p}$ and $\tau$ by picking $\gamma$ close enough to $d=2$.)
\end{proof}

As a last auxiliary result of potentially independent interest, we note the following remarkably simple
embedding which holds true for any bounded open set without further assumptions on its geometry:
\begin{lemma}
  \label{lem:interpol-infty-hoelder}
  Let $\alpha > 0$. Then $(L^\infty(\Omega),C^\alpha(\Omega))_{\theta,1} \embeds
  C^{\alpha\theta}(\Omega)$ for any $\theta \in (0,1)$.
\end{lemma}
\begin{proof}
  Let
  $u \in C^\alpha(\Omega)$ and estimate
  \begin{align*}
    \sup_{\substack{\x, \y \in \Omega\\\x\neq \y}}
    \frac{|u(\x) - u(\y)|}{|\x -\y |^{\alpha \theta }} &\le
    \sup_{\substack{\x, \y \in \Omega\\\x\neq \y}} |u(\x) - u(\y)|^{1-\theta} \sup_{\substack{\x, \y \in \Omega\\\x\neq \y}} \frac{|u(\x) - u(\y)|^\theta}{|\x -\y |^{\alpha \theta}} \\ 
    &\le \bigl (2 \|u\|_{L^\infty(\Omega)})^{1-\theta} \|u \|_{C^\alpha(\Omega)}^\theta.
  \end{align*}
  Together with an obvious estimate for $\sup_{\x \in \Omega} |u(\x)|$ one gets, for every
  $u \in C^\alpha(\Omega)$,
  \[ \|u \|_{C^{\alpha \theta}(\Omega)} \le 3 \|u \|^{1-\theta}_{L^\infty(\Omega)} \|u
    \|^{\theta}_{C^\alpha(\Omega)}.
  \]
  Thus, referring to~\cite[Lem.~1.10.1]{Tri}, $C^{\alpha\theta}(\Omega)$ is of class
  $J(\theta)$ with respect to $L^\infty(\Omega)$ and $C^\alpha(\Omega)$ from which we obtain the
  desired embedding.
  % \begin{equation*}%\label{e-12}
  %   \bigl( L^\infty(\Omega), C^\alpha(\Omega) \bigr)_{\theta,1} \hookrightarrow C^{\alpha \theta}(\Omega).\qedhere
  % \end{equation*}
\end{proof}
% We apply this result in our situation.\footnote{Hier ist noch zu kl\"aren, warum wir mit unseren Operatoren in der Situation von Arendt sind. Teilweise 2.1 und 2.4, aber was ist z.B. mit $p < 2$?}
% \begin{corollary} \label{c-2}
% Let Assumption~\ref{a-general}\ref{a-general-1} be fulfilled and let $p > d$.
% Then, for every $\tau \in \bigl( \frac {d}{2p}, 1 \bigr]$ one has $(A+1)^{-\tau} \in \mathcal L(L^p;L^\infty)$ and for all $u \in \dom \bigl( (A_p+1)^\tau \bigr)$
% \[ \|u \|_{L^\infty} \le C \|(A+1)^\tau u \|_{L^p} \le C \|u\|_{\dom((A_p+1)^\tau)}.
% \]
% \end{corollary}
% \begin{proof}
% In case of $d>2$, thanks to Remark~\ref{r-p2}, one may take $\gamma =d$ in Corollary~\ref{c-1}. If $d = 2$ the same lemma gives the embedding $\VV = W^{1,2}_D(\Omega) \hookrightarrow L^r(\Omega)$ for every finite $r$. Thus $\gamma $ may be taken arbitrarily close to $2$ and the assertion also follows from Corollary~\ref{c-1}.
% \end{proof}
%
\begin{proof}[Proof of Theorem~\ref{t-main1}]
  Set $\theta \coloneqq \kappa/\alpha\in (0,1)$ and
  $\sigma \in \bigl(\frac12+\frac{d}{2q} + \theta(\frac12 - \frac{d}{2q}), 1\bigr)$ as in the
  theorem. A short computation shows that we can write
  $\sigma = (1-\theta)(\frac12 + \tau) + \theta$ with some $\tau \in
  (\frac{d}{2q},\frac12)$. Thus, the reiteration theorem~(\cite[Thm.~1.10.2]{Tri}) implies that
 \begin{equation*}
   \bigl( W_D^{-1,q}(\Omega), \dom(\cA_q+1) \bigr)_{\sigma, 1}= \Bigl( \bigl(
   W_D^{-1,q}(\Omega), \dom(\cA_q+1) \bigr)_{\frac 12 + \tau,1}, \dom(\cA_q+1) \Bigr)_{\theta,
     1}. 
 \end{equation*}
 We show that the first space on the right embeds continuously into $L^\infty(\Omega)$. Indeed,
 by interpolation for fractional power domains of so-called positive operators as
 in~\cite[Thm.~1.15.2]{Tri}, we have
 \begin{equation*}
   \bigl( W_D^{-1,q}(\Omega), \dom(\cA_q+1) \bigr)_{\frac 12 + \tau,1} \embeds \dom((\cA_q+1)^{1/2+\tau}).
 \end{equation*}
 % Thus,
 % \begin{equation*}
 %    \bigl( W_D^{-1,q}(\Omega), \dom(\cA_q+1) \bigr)_{\sigma, 1} \embeds \Bigl( \dom\bigl((\cA_q+1)^{1/2 +\tau}\bigr), \dom(\cA_q+1) \Bigr)_{\theta,1}.
 % \end{equation*}
 But for $\tau \in (\frac{d}{2q},\frac12)$, by combining
 Proposition~\ref{p-propAq}~\ref{p-propAq3}---this is the point where we need
 Assumption~\ref{a-general}~\ref{a-general-2}---and Corollary~\ref{c-2}, we find
  \[ \dom \bigl( (\cA_q+1)^{1/2 +\tau} \bigr) = \dom \bigl( (A_q+1)^{\tau} \bigr)
    \hookrightarrow L^\infty(\Omega).
  \]
  By assumption, the restriction of the foregoing embedding to $\dom(\cA_q+1)$ is precisely
  $\dom(\cA_q+1) \hookrightarrow C^\alpha(\Omega)$. Interpolating these and using
  Lemma~\ref{lem:interpol-infty-hoelder}, we find
  \[ \bigl( W_D^{-1,q}(\Omega), \dom(\cA_q+1) \bigr)_{\sigma, 1} \hookrightarrow \bigl(
    L^\infty(\Omega), C^\alpha(\Omega) \bigr)_{\theta,1} \embeds C^{\alpha \theta}(\Omega)
  \]
  and this was the claim, since $\alpha\theta = \kappa$.

  Now the embedding for $\dom((\cA_q + 1)^\sigma)$ itself follows easily by squeezing $s$ between
  $\frac12 + \frac{d}{2q} + \frac\kappa\alpha(\frac12-\frac{d}{2q})$ and $\sigma$ and using the
  previous part via~\cite[Thms.~1.3.3 and~1.15.2]{Tri}:
  \begin{align*}    
 \dom ((\cA_q+1)^{\sigma}) & \embeds \bigl( W^{-1,q}_D(\Omega), \dom(\cA_q+1)
 \bigr)_{\sigma,\infty}\\ & \hookrightarrow \bigl(
 W^{-1,q}_D(\Omega), \dom(\cA_q+1) \bigr)_{s,1} \hookrightarrow C^\kappa(\Omega). \qedhere
\end{align*}
\end{proof}

\subsection*{The domain of $\cA_q$} In the above proof, we have worked only with $\cA_q + 1$ to have an invertible operator at hand
which is much more convenient. However, the \emph{sets} $\dom(\cA_q)$ and $\dom(\cA_q+1)$ are
always the same, and if $\cA$ is continuously invertible, then so is $\cA_q$ and it follows that $\dom(\cA_q)$
and $\dom(\cA_q+1)$ are also equivalent as Banach spaces, each equipped with the respective
graph norm. This transfers to the domains of their fractional powers as well.

By the Lax-Milgram lemma, the operator $\cA$ in turn is continuously invertible whenever we have
a Poincar{\'e} inequality for $V$ at hand. For the latter it is enough to establish that nonzero
constant functions do not belong to $V$. Within our geometric setup of
Assumption~\ref{a-general}, this is already guaranteed by either
$D \cap \overline N \neq \emptyset$, so the Dirichlet- and Neumann boundary parts share a common
interface, or by $D$ containing at least one (relatively) inner point. See for
instance~\cite[Lemma~7.3]{hardy}. (In fact, in the former case it is already enough to have
Lipschitz charts for all points in the relative boundary $\partial D$ within $\partial\Omega$ at
hand; cf.~\cite[Sect.~6]{CMR}.)

In this sense, the statement for $\cA_q+1$ in Theorem~\ref{t-main1} can be immediately
transferred to $\cA_q$ whenever the geometry assumptions admit a Poincar\'e inequality for $V$.

\section{H\"older properties for  $\dom(\cA_q+1)$} \label{s-Hoelderdomai}
In the main result of Section~\ref{sec-Embed} the embedding of $\dom(\cA_q+1)$ into some
H\"older space was a given. We now turn to the question when such an embedding is true. A very
general answer was given in \cite[Theorem~1.1]{ER}, where the result in Theorem~\ref{t-result}
below was proved for all space dimensions $d$. This proof is extremely involved, the natural
instruments being Sobolev-Campanato spaces and De~Giorgi estimates.

However, for dimensions up to $4$ one can avoid this machinery and base the arguments only on
the classical Ladyshenskaya result on H\"older continuity for solutions of the pure Dirichlet
problem, see Proposition~\ref{p-lady} below, and some more elementary yet intricate technical
means. This is what we will carry out here. It will be a welcome byproduct of the present
approach that we easily obtain a uniform result with respect to the given geometry and the
$L^\infty(\Omega)$-bound and ellipticity constant of the coefficient function $\mu$.

In order to formulate our main result of this section, we introduce two more 
geometric conditions; the first one relies on the rather classical notion with a twist of saying
that an open subset $\Lambda$ of $\R^d$ is \emph{of class $(A_\gamma)$} (at $\Upsilon \subseteq
\partial\Lambda$) with a constant $\gamma > 0$, if
\begin{equation*}
  \lambda_d \bigl( B_r(\x) \setminus \Lambda \bigr) \ge \gamma \lambda_d \bigl( B_r(\x)
  \bigr) \qquad \text{for all}~\x\in\Upsilon,~r \in (0,1].\tag{$A_\gamma$}
\end{equation*}
Of course, necessarily $\gamma < 1$. This condition prevents inwards cusps of $\Lambda$ at $\Upsilon$. If $\Upsilon =
\partial\Lambda$, we just refer to $\Lambda$ being of class $(A_\gamma)$. 
The second condition, rather intriguing, concerns the
interface between the Dirichlet boundary part $D$ and the Neumann boundary part
$N = \partial \Omega \setminus D$ in the boundary of $\Omega$:
\begin{assu}\label{a-Interface} We consider the following further geometric assumptions for $\Omega$ and $D$:
 \begin{enumerate}[(a)]
  \item\label{a-Interface:i} There is some $\gamma \in (0,1)$ such that $\Omega$ is of class
    $(A_\gamma)$ at $D$. 
\item\label{a-Interface:ii} Using the notation of Assumption~\ref{a-general}~\ref{a-general-1},
  there are two constants $c_0 \in (0,1)$ and $c_1 > 0$ such that for any point
  $\x \in E \coloneqq D \cap \overline{N}$, every $ y \in \R^{d-1}$ such that
  $( \y,0) \in \phi_\x (E \cap V_\x)$ and every $s \in (0,1]$ it holds
%  \[
%   \lambda_{d-1} \bigl( \bigl\{ \tilde{\z} \in \lowdim{B}_s(\tilde{\y}) \colon \dist \bigl( \tilde{\z}, \phi_\x ( N \cap V_\x )\bigr) > c_0 s \bigr\} \bigr) \ge c_1 s^{d-1}.
% \]
 \[
  \lambda_{d-1} \Bigl( \bigl\{ {\z} \in  \lowdim{B}_s(\y) \colon \dist \bigl( \z, \phi_\x ( N \cap V_\x )\bigr) > c_0 s \bigr\} \Bigr) \ge c_1 s^{d-1}. 
\]
Here and in the sequel, $\lowdim{B}_r({\y})$ denotes the open ball of radius $r$ in $\R^{d-1}$
with its center at 
${\y} \in \R^{d-1}$, and in the distance function we tacitly consider $\phi_x(N \cap V_\x) \subset [z_d =
0]$ as a subset of $\R^{d-1}$ in the obvious manner.
  \end{enumerate}
\end{assu}
%
% \begin{rem}
%  Note that in the case where $D \cap \overline{N} \neq \emptyset$,
%  Assumption~\ref{a-Interface}~\ref{a-Interface:ii} already implies that $D \cap \bigl(
%  \bigcup_{\x \in N} V_{\x} \bigr)$ is not negligible  with respect to the boundary measure on
%  $\partial \Omega$. Thus the constant functions do not belong to $\VV=W^{1,2}_D(\Omega)$ in this
%  case and Assumption~\ref{a-general}~\ref{a-general-3} is satisfied. We nevertheless continue to
%  notice whenever the condition in Assumption~\ref{a-general}~\ref{a-general-3} is needed, to
%  cover also the cases where the Dirichlet and the Neumann part of the boundary do not meet. 
% \end{rem}

We can now formulate the main theorem of this section.
\begin{theorem} \label{t-result} Suppose that $\Omega$ and $D$ satisfy
  Assumption~\ref{a-general}~\ref{a-general-1} and Assumption~\ref{a-Interface}, and let $q>d$
  with $d \in \{2,3,4\}$. If $d=4$, suppose also that Assumption~\ref{a-general}~\ref{a-general-2} is
  satisfied. Then there is an $\alpha > 0$ such that for every $f \in W^{-1,q}_D(\Omega)$ the
  equation
 \begin{equation} \label{e-euq}
  (\cA_q+1) v = f
 \end{equation}
 has a unique solution $v \in W^{1,2}_D(\Omega)$ that belongs to the H\"older space
 $C^\alpha(\Omega)$. Moreover, the mapping
 $W^{-1,q}_D(\Omega) \ni f \mapsto v \in C^\alpha(\Omega)$ is continuous and its norm depends
 only on the geometry of $\Omega$ and the $L^\infty(\Lambda)$-bound and ellipticity constant of
 $\mu$.
\end{theorem}
\begin{rem} \label{r-lp/2} We comment on Theorem~\ref{t-result}.
  \begin{enumerate}[(a)]
  \item It is well known that, in general, the condition $q > d$ is already necessary for the
    boundedness of the solution, see~\cite[Ch.~I.2]{lady}.
  \item It is easily seen that if $f \in L^p(\Omega)$ with $p>d/2$, then also
    $f \in W^{-1,q}_D(\Omega)$ where $q = 2p > d$ with continuous embedding thanks to
    Remark~\ref{r-p2}. In this sense, Theorem~\ref{t-result} is also a result on H\"older
    regularity for the operators $A_p + 1$ for $p>d/2$. (Note that so far we had only seen that
    the $L^p(\Omega)$-solution to~\eqref{e-euq} is in $L^\infty(\Omega)$ via ultracontractivity
    as in Corollary~\ref{c-2}---but this was already true for a fractional power of $A_p+1$ and
    so some opportunity for improvement for $A_p+1$ itself was expected.)
  \end{enumerate}
\end{rem}
%\begin{rem} \label{r-proper}
%The suppositions on $\Gamma$, which plays the role of the 'Neumann boundary part' 
%(see Remark \ref{r-randbed} above) are motivated as follows: If $\Gamma = \emptyset$, then
%the whole setting leads to the pure Dirichlet problem for which the H\"older continuity is
%known since long, see \cite[Ch.~II Thm.B.4]{kind}. If $\Gamma =\partial \Omega$, then, under
%our assumptions made below, one is confronted with the pure Neumann problem on a Lipschitz
%domain. In this case, the H\"older continuity for the solution is proved in \cite{h/m/r/s}.
%\textbf{ ggf. anderes Zitat, Hoeppner}
%Thus, we really restrict us here to mixed boundary conditions.
%\end{rem}
%
%
%
%
\label{p-Ablaufplan}Let us sketch an outline for the proof of Theorem~\ref{t-result}. We will
rely on the classical 
techniques of localization, transformation and reflection to tackle~\eqref{e-euq} in the form of
a finite number of similar problems on model sets with a very particular geometry. For these we
will rely on classical H\"older regularity results of Ladyzhenskaja or Kinderlehrer which base
on variants of Assumption~\ref{a-Interface}~\ref{a-general-1} . The treatment of local problems
at the pure Dirichlet part $D \setminus \overline N$ will be quite immediate due to
Assumption~\ref{a-Interface}~\ref{a-Interface:i}, and we will also be able to transfer the
Neumann boundary part $N = \partial\Omega\setminus D$ to the pure Dirichlet situation via
Assumption~\ref{a-general}~\ref{a-general-1} and reflection techniques. Of course, the most
interesting part will be the interface $D \cap \overline N$ with
Assumption~\ref{a-Interface}~\ref{a-Interface:ii}. The intriguing idea here is that
Assumption~\ref{a-Interface}~\ref{a-Interface:ii} will allow to transform the localized problem
once more in a particular way such that the resulting set will in fact be amendable by
Assumption~\ref{a-Interface}~\ref{a-Interface:i}.

%%%%%%%%%%%%%%%%%%%%%%%%%%%%%%%%%%%%%%%%%%%%%%%%%%%%%%%%%%%%%%%%%%%%%%%%%%%%%
%
%
%
%
\subsection{Localization and transformation techniques}\label{sec:local-transf-techn}
%
%
%
%
%%%%%%%%%%%%%%%%%%%%%%%%%%%%%%%%%%%%%%%%%%%%%%%%%%%%%%%%%%%%%%%%%%%%%%%%%%%%%%
%
%
%
%
%\noindent
In this subsection we recall, for the reader's convenience, some technical results on
localization and transformation techniques for~\eqref{e-euq} which are needed later on. For all
the following considerations the coefficient function $\mu$ is considered as in
Section~\ref{ss-prelim}; in particular it is elliptic with constant $\kappa_{\text{ell}}$.

We start by quoting a classical theorem
(see~\cite[Ch.~II Appendix~B/C]{kind}) on the H\"older continuity for the
solution of the Dirichlet problem. The result is formulated for a generic bounded domain $\Lambda \subset
\R^d$ since we will use it for several local model sets in the proof of Theorem~\ref{t-result};
the definitions of $\mu$ and $\cA$ are to be understood \emph{mutatis mutandis}.
\begin{proposition} \label{p-lady}
 Let $\Lambda \subset \R^d$ be a bounded domain and let $v \in W^{1,2}_0(\Lambda)$ be the solution of
 \begin{equation} \label{e-gleichung}
  \cA v = f_0 + \sum_{j=1}^d
        \frac{\partial f_j}{\partial x_j},
 \end{equation}
 where $f_0, f_1, \ldots, f_d \in L^q(\Lambda)$ with $q>d$ and $\frac{\partial}{\partial
   x_j}$ denotes the distributional derivative. Then the following
 holds true.
 \begin{enumerate}[(a)]
  \item The function $v$ admits a bound
        \begin{equation} \label{e-Linfyabsch}
          \|v\|_{L^\infty(\Lambda)} \le c  \;
                % \frac{\lambda_d(\Lambda)^{\frac{1}{d} - \frac {1}{q}}}{\kappa_{\mathrm{ell}}} \;
                \sum_{j=0}^d \|f_j\|_{L^q(\Lambda)}.
        \end{equation}
  \item Suppose that there exists  $\gamma \in (0,1)$ such that $\Lambda$ is of class
    $(A_\gamma)$. Then $v$ is H\"older-continuous, more precisely: there is an $\alpha \in (0,1)$ independent of $f_0, f_1, \ldots,f_d$ such that
        \begin{equation} \label{e-Hoie}
          \sup_{\x, \y \in B_r(\z) \cap \Lambda}
                |v(\x) - v(\y)| \le c \; \sum_{j=0}^d \|f_j\|_{L^q(\Lambda)}
                \; r^\alpha
        \end{equation}
        holds true for all $\z \in \R^d$ and $r > 0$.
      \end{enumerate}
      In both estimates~\eqref{e-Linfyabsch} and~\eqref{e-Hoie}, the constant depends only on
      the geometry of $\Lambda$ and the $L^\infty(\Lambda)$-bound and ellipticity constant of
      $\mu$.
\end{proposition} 
\begin{rem} \label{r-interpret}
% \begin{enumerate}[(a)]
% \item\label{r-interpret:i}
  The right hand side of \eqref{e-gleichung} is to be understood as the antilinear form
 \[ W^{1,q'}_0(\Lambda) \ni \psi \mapsto \int_\Lambda f_0\overline{\psi} -
        \sum_{j=1}^d f_j \frac {\partial \overline{\psi}}{\partial x_j} 
 \]
 which clearly belongs to $W^{-1,q}(\Lambda) \hookrightarrow
 W^{-1,2}(\Lambda)$. Thus, the uniqueness of the solution $v$ follows from the ellipticity of
 $\ft$ and the Lax-Milgram lemma.
 % \item\label{r-interpret:ii}
 
 On the other hand, while every antilinear form in $W^{-1,q}(\Lambda)$ can be represented in the
 foregoing form, this representation is in general non-unique. But it is
 in fact well known that $W^{-1,q}(\Lambda)$ is isometrically isomorphic to the quotient space
 with respect to such representations; see \cite[Ch.~1.1.14]{mazyasob}. Hence, taking the
 infimum over all representing families in the estimates~\eqref{e-Linfyabsch}
 and~\eqref{e-Hoie}, in the setting of Proposition~\ref{p-lady} one obtains the continuity of
 \[ \cA_q^{-1} \colon W^{-1,q}(\Lambda) \to
        C^\alpha(\Lambda).
 \] 
 The norm of this mapping depends only on the geometry of $\Lambda$ and the $L^\infty(\Lambda)$-bound and
 ellipticity constant of $\mu$.
 %\end{enumerate}
\end{rem}
The following extrapolation of the Lax-Milgram isomorphism will give us the small $\eps$ in regularity that allows us to treat also the case of dimension four.
\begin{proposition}[\protect{\cite[Thm~5.6]{HJKR}}]
\label{p-groegerrepr}
 Let Assumptions~\ref{a-general}~\ref{a-general-1} and~\ref{a-general-2} be satisfied. Then
 there is an $\eps > 0$ such that $\dom(\cA_q + 1) = W^{1,q}_D(\Omega)$ for all $q \in [2,2+\eps)$, that is, the operator
 \[  \cA_q +1 \colon W^{1,q}_D(\Omega) \to W^{-1,q}_D(\Omega)
 \]
 is a topological isomorphism. The norms of $(\cA_q+1)^{-1}$ are uniform with respect to $\eps$
 and the  $L^\infty(\Omega)$-bound and
 ellipticity constant of $\mu$.
\end{proposition} 
%

% The plan to prove Theorem~\ref{t-result}\label{p-Ablaufplan} will be to first localise the
% problem. For points in the (relative) interior of the Dirichlet boundary part we can immediately
% apply Proposition~\ref{p-lady}. Around points from the closure of the Neumann boundary part we
% will transform the localised version of Equation~\eqref{e-euq} with several bi-Lipschitz
% transforms until we get to a situation, where the domain is in a lower half-space and the
% Neumann boundary part is exactly the intersection of its boundary with the boundary of the
% half-space. Then we can do an even reflection along the boundary of the half-space, ending up
% with a pure Dirichlet problem, that, again by Proposition~\ref{p-lady}, has a
% H\"older-continuous solution.

The plan how we aim to prove Theorem~\ref{t-result} was already sketched above. We now have seen the
main tool with which we leverage H\"older-continuity for the localized and transformed problems in the form of Proposition~\ref{p-lady}. It
remains to make sure that the localization, transformation and possibly reflection techniques
are compatible with Proposition~\ref{p-lady}; this concerns continuity for the associated mappings between the function spaces involved and of course in particular the
assumption in the domain in Proposition~\ref{p-lady} for the actual H\"older estimate.

% In order to make this plan work, we have to make sure that all our transformations do not alter the H\"older-continuity of the solutions, do not destroy the regularity properties of the domains and give rise to continuous mappings between the occuring function spaces.
This we will do in the following series of technical lemmas. We start with three of them that
deal with the localization. Recall the notation $N = \partial \Omega \setminus D$ for the
Neumann boundary part. First, we deal with localized Sobolev functions with partially vanishing trace.
\begin{lemma}[\protect{\cite[Ch.~4.2]{HaR}}] \label{l-restr/ext}
%  \footnote{in \cite{HaR} steht in den Voraussetztungen jeweils '$\Omega
%  \cup \Gamma$ regular' drin. Das wird im Beweis wohl nicht gebraucht, das
%  sollte man hier aber mal bemerken. \begin{color} {red} nein, wird nicht gebraucht. ich hoffe, ich finde 
% eine neuere Quelle, wo wir das vern\"unftig aufgeschrieben haben \end{color}}
 Let $U \subseteq \R^d$ be open and set $\Omega_\bullet \coloneqq  \Omega
 \cap U$ as well as $D_\bullet \coloneqq  \partial \Omega_\bullet \setminus N$. Fix an arbitrary function $\eta \in C^\infty_c(\R^d)$ with $\supp(\eta) \subseteq U$. Then for any $q \in (1, \infty)$ we have the following assertions:
 \begin{enumerate}[(a)]
  \item \label{l-restr/ext:enum-i} If $v \in W^{1,q}_D(\Omega)$, then 
        $\eta v|_{\Omega_\bullet} \in W^{1,q}_{D_\bullet}(\Omega_\bullet)$.
  \item Denote by $E_0$ the zero extension operator % Let for any $v \in L^1(\Omega_\bullet) $ the symbol $\tilde v$ indicate
        % the  extension of $v$ to $\Omega$ by zero.
        % Then $E_0$ induces a continuous linear mapping
      %   \[ W^{1,q}_{D_\bullet}(\Omega_\bullet) \ni v \mapsto
      %           E_0 (\eta v) \in W^{1,q}_D (\Omega).
      %   \]
      %   %has its image in and is continuous.
      % \item
     \label{l-project:i}
    and let $f \in W^{-1,q}_D(\Omega)$. Then $ f \mapsto f_\bullet$ with
        \[f_\bullet \colon w \mapsto \bigl\langle f, E_0(\eta w) \bigr\rangle,
          \quad w \in W^{1,q'}_{D_\bullet}(\Omega_\bullet)
        \]
        defines a continuous linear operator
        $W^{-1,q}_D(\Omega) \to W^{-1,q}_{D_\bullet}(\Omega_\bullet)$.
 \end{enumerate}
\end{lemma}
The next lemma is about the localization of a solution $v$ to the elliptic equation
$(\cA + 1)v = f$ and the 'localized' equation. Here and also in the following, we will need
several versions of the divergence-gradient type operators $\cA$ with different underlying
spatial sets, coefficient functions and associated Sobolev spaces respecting partially vanishing
trace conditions. We will use the notation $-\nabla \cdot\eta \nabla$ with the coefficient function
$\eta$ for these. It will always be clear from the context which precise incarnation is meant.

\begin{lemma}[\protect{\cite[Lem.~4.7]{HaR}}] \label{l-project}
 Let $U$, $\eta$, $\Omega_\bullet$ and $D_\bullet$ be as in the foregoing lemma. Set
 $\mu_\bullet \coloneqq \mu|_{\Omega_\bullet}$ and consider the
 operator $- \nabla \cdot \mu_\bullet \nabla \colon
 W^{1,2}_{D_\bullet}(\Omega_\bullet) \to W^{-1,2}_{D_\bullet}(\Omega_\bullet)$. Let $f \in W_D^{-1,2}(\Omega)$ and let $v \in
 W^{1,2}_D(\Omega)$ be the solution of $(\cA+1)v = f$.
 %
 % \begin{equation} \label{e-glei1}
 %   \cA v = f \in W_D^{-1,2}(\Omega).
 % \end{equation}
 %
 % Then the following holds true.
 % \begin{enumerate}[(a)]
 % \item \label{l-project:i}For all $q \in (1, \infty)$ the linear form
 %        %
 %   \[ f_\bullet : w \mapsto \langle f, \widetilde {\eta w} \rangle,
 %   \]
 %        %
 %   where the tilde again means the extension by zero to the whole of $\Omega$, is well defined
 %   and continuous on $W^{1,q'}_{D_\bullet}(\Omega_\bullet)$, whenever
 %   $f \in W^{-1,q}_D(\Omega)$.
 % \item \label{l-project:ii}
 Then $u \coloneqq \eta v|_{\Omega_\bullet} $ satisfies
 \begin{equation} \label{e-locglei} - \nabla \cdot \mu_\bullet \nabla u = f^\bullet \coloneqq
   f_\bullet - \nabla\cdot v\mu_\bullet\nabla\eta  - \mu_\bullet \nabla v|_{\Omega_\bullet} \cdot \nabla \eta|_{\Omega_\bullet}
   - \eta v|_{\Omega_\bullet}
   \quad \text{in} \quad W^{-1,2}_{D_\bullet}(\Omega_\bullet).
 \end{equation}
 % where $T_v$ is defined by
 % \[ T_v \colon \quad W^{1,2}_{D_\bullet}(\Omega_\bullet) \ni w \mapsto \int_{\Omega_\bullet} v
 %   \mu_\bullet \nabla \eta \cdot \nabla \overline{w}.
 % \]
 % by $T_v$, then $u \coloneqq \eta v|_{\Omega_\bullet} $ satisfies
 % %
 % \begin{equation} \label{e-locglei} - \nabla \cdot \mu_\bullet \nabla u =f_\bullet + T_v -
 %   \mu_\bullet \nabla v|_{\Omega_\bullet} \cdot \nabla \eta|_{\Omega_\bullet} =: f^\bullet.
 % \end{equation}
 % \end{enumerate}
\end{lemma}
Note that $D_\bullet$ will always be a nontrivial boundary part of $\Omega_\bullet$ due to the
localization procedure as established in Lemma~\ref{l-restr/ext}. It is thus convenient to
consider the localized problem without a zero-order term as in~\eqref{e-locglei}, since this is
ultimately also the form about which Proposition~\ref{p-lady} makes a statement.      
%
%
% \begin{proof}
% \begin{enumerate}[(a)]
%  \item The mapping $f \mapsto f_\bullet$ is the adjoint to $v \mapsto
%         \widetilde{\eta v}$, which maps by the preceding lemma
%         $W^{1,q'}_{D_\bullet}(\Omega_\bullet)$ continuously into
%         $W^{1,q'}_D(\Omega)$.
%  \item The assertion is proved in \cite[Lemma~4.7]{HaR}.
%         \qedhere
%  \end{enumerate}
% \end{proof}
%
%
\begin{lemma} \label{e-abschfbullet} Let Assumptions~\ref{a-general}~\ref{a-general-1} be
  satisfied; if $d=4$, let also Assumption~\ref{a-general}~\ref{a-general-2} hold true. Take
  $U$, $\eta$, $\Omega_\bullet$ and $D_\bullet$ as in Lemma~\ref{l-restr/ext}. Let further
  $f \in W^{-1,q}_D(\Omega)$ for some $q > d$ and consider $f^\bullet$, defined as in
  Lemma~\ref{l-project} via $(\cA+1)v = f$. Furthermore, assume that there is a linear extension
  operator $\mathfrak E_\bullet$ which acts continuously from
  $W^{1,r}_{D_\bullet}(\Omega_\bullet)$ into $W^{1,r}(\R^d)$ for $r \in (1,\infty)$. Then there
  exists $p>d$ such that $f^\bullet \in W^{-1,p}_{D_\bullet}(\Omega_\bullet)$, and the mapping
  $W^{-1,q}_D(\Omega) \ni f \mapsto f^\bullet \in W^{-1,p}_{D_\bullet}(\Omega_\bullet)$ is
  continuous.
 % \begin{enumerate}[(a)]
 %  \item Assume $d = 2$ and $q \ge 2$ or $d = 3$ and $q \in [2,6]$. Then $f^\bullet \in W^{-1,q}_{D_\bullet}(\Omega_\bullet)$, and the mapping $W^{-1,q}_D(\Omega) \ni f \mapsto f^\bullet \in W^{-1,q}_{D_\bullet}(\Omega_\bullet)$ is continuous.
 %  \item Assume $d = 4$ and $q > 4$. Then there is a $p > 4$ such that for every $f \in W^{-1,q}_D(\Omega)$ one has $f^\bullet \in W^{-1,p}_{D_\bullet}(\Omega_\bullet)$. Moreover, the mapping $W^{-1,q}_D(\Omega) \ni f \mapsto f^\bullet \in W^{-1,p}_{D_\bullet}(\Omega_\bullet)$ is continuous.
 % \end{enumerate}
\end{lemma}
\begin{proof}
  Let us first recall that the usual Sobolev embeddings hold for both $\Omega$ and
  $\Omega_\bullet$, respectively, cf.~Remark~\ref{r-p2}. Now, let us consider the terms in the right hand side of~\eqref{e-locglei}, so the definition of
 $f^\bullet$, from left to right. We have $f_\bullet \in W^{-1,q}_{D_\bullet}(\Omega_\bullet)$
 depending continuously on $f \in W^{-1,q}_D(\Omega)$
 thanks to Lemma~\ref{l-restr/ext}~\ref{l-project:i}, so this term is fine without further
 ado. For the remaining terms, we distinguish between $d=2,3$ and $d=4$, starting with the
 former.  We note that the proof of the continuity of $f\mapsto f_\bullet$ is implicitly contained in the following estimates.
% \begin{enumerate}[(a)]
% \item

 Let first $d=2,3$.  Due to the Lax-Milgram lemma and Sobolev embedding, we have
 \begin{equation} \label{e-estim}
  \|v\|_{W^{1,2}_D (\Omega)} \le c \|f\|_{W^{-1,2}_D (\Omega)} \le c \|f\|_{W^{-1,q}_D (\Omega)}
\end{equation}
where $c$ only depends on geometry and the ellipticity constant of $\mu$. Concerning $- \nabla\cdot v\mu_\bullet\nabla\eta$, for any $p \in [1,\infty]$ we have the estimate
   \begin{equation}\label{e-ab0sch}
     \bigl| \langle - \nabla\cdot v\mu_\bullet\nabla\eta, w \rangle \bigr|
     \le \|v \|_{L^p(\Omega_\bullet)} \; \| \mu
     \|_{L^\infty(\Omega; \R^{d\times d})} \; \| \nabla \eta
     \|_{L^\infty(\Omega_\bullet)} \; \|w
     \|_{W_{D_\bullet}^{1,p'}(\Omega_\bullet)}
   \end{equation}
  In particular, for $p=\min(q,6) >d$, we find
   \begin{equation*}
     \bigl| \langle
            - \nabla\cdot v\mu_\bullet\nabla\eta, w \rangle \bigr|\le c \|f \|_{W^{-1,q}_D(\Omega)} \; \| \mu
            \|_{L^\infty(\Omega; \R^{d\times d})} \; \| \nabla \eta
            \|_{L^\infty(\Omega_\bullet)} \; \|w
            \|_{W_{D_\bullet}^{1,p'}(\Omega_\bullet)}%\nonumber
  \end{equation*}
  thanks to the Sobolev embedding $W^{1,2}_D (\Omega) \hookrightarrow L^6(\Omega)\embeds L^p(\Omega)$ and
  estimate~\eqref{e-estim}. Thus, $- \nabla\cdot v\mu_\bullet\nabla\eta \in W^{-1,p}_{D_\bullet}(\Omega_\bullet)$. The same
  argument and~\eqref{e-estim}
  moreover shows that $\eta v|_{\Omega_\bullet} \in L^p(\Omega_\bullet) \embeds
  W^{-1,p}_{D_\bullet}(\Omega_\bullet)$. 

 Concerning the term $\mu_\bullet \nabla v|_{\Omega_\bullet} \cdot \nabla
 \eta|_{\Omega_\bullet}$, it is easily observed that if $v \in W^{1,r}_D(\Omega)$, then the term belongs to $L^r(\Omega)$ with the estimate
 \begin{equation}
  \|\mu_\bullet \nabla v|_{\Omega_\bullet} \cdot \nabla \eta|_{\Omega_\bullet}\|_{L^r(\Omega)}
  \le \|\mu\|_{L^\infty(\Omega; \R^{d\times d})} \|\nabla \eta\|_{L^\infty(\Omega)}
  \|v\|_{W^{1,r}_D (\Omega)} \label{eq:e-ab0sch2}.
\end{equation}
In particular, for $r=2$, we obtain via~\eqref{e-estim}:
\begin{equation*}
   \|\mu_\bullet \nabla v|_{\Omega_\bullet} \cdot \nabla \eta|_{\Omega_\bullet}\|_{L^2(\Omega)}\leq c \|\mu\|_{L^\infty(\Omega; \R^{d \times d})}\|\nabla \eta\|_{L^\infty(\Omega)} \|f\|_{W^{-1,q}_D (\Omega)}.
 \end{equation*}
 Thus, with the same choice for $p$ as before,  $\mu_\bullet \nabla v|_{\Omega_\bullet} \cdot \nabla \eta|_{\Omega_\bullet} \in
 W^{-1,p}_D(\Omega_\bullet)$ due to the embedding $L^2(\Omega_\bullet) \hookrightarrow
 W^{-1,p}_{D_\bullet}(\Omega_\bullet)$.
 % , which holds true for any $p\geq 1$ if $d=2$ and for $p \in
 %[2,6]$ if $d=3$. % (Note that one has the required embedding at hand due to our supposition on the existence of the extension operator $\mathfrak E_\bullet$.)

 Now let $d=4$. Thanks to Proposition~\ref{p-groegerrepr}, there is an $\eps>0$ such
   that $v \in W^{1,2+\eps}_D(\Omega)$ with the estimate
   \begin{equation}
     \|v\|_{W^{1,2+\eps}_D(\Omega)} \le c \|f\|_{W^{-1,2+\eps}_D(\Omega)} \le c\|f\|_{W^{-1,q}_D(\Omega)}.\label{eq:groeger-estimate}
   \end{equation}
   Having this at hand, for the estimate of the term $- \nabla\cdot v\mu_\bullet\nabla\eta$ we again exploit~\eqref{e-ab0sch},
   this time taking $p = 4\cdot \frac
   {2+\eps}{2-\eps}$ such that precisely
   $W^{1,2+\eps}_D(\Omega) \hookrightarrow L^p(\Omega)$. Note that $p>4=d$. Again, it follows
   analogously, this time via~\eqref{eq:groeger-estimate}, that $\eta v|_{\Omega_\bullet} \in
   L^p(\Omega_\bullet) \embeds W^{-1,p}_{D_\bullet}(\Omega_\bullet)$.  

   Finally, we estimate 
   again as in~\eqref{eq:e-ab0sch2} but pick $r=2+\eps$ and consider~\eqref{eq:groeger-estimate} to observe
   $\mu_\bullet \nabla v|_{\Omega_\bullet} \cdot \nabla \eta|_{\Omega_\bullet} \in
   L^{2+\eps}(\Omega_\bullet)$ together
   with the estimate
 \begin{equation*}
\|\mu_\bullet \nabla v|_{\Omega_\bullet} \cdot \nabla \eta|_{\Omega_\bullet}\|_{L^{2+\eps}(\Omega)} \le c \|\mu\|_{L^\infty(\Omega; \R^{d \times d})} \|\nabla \eta\|_{L^\infty(\Omega)} \|f\|_{W^{-1,q}_D(\Omega)}.
\end{equation*}
 With $p$ as before, one has the embedding
   $L^{2 + \eps}(\Omega_\bullet) \hookrightarrow W^{-1,p}_{D_\bullet}(\Omega_\bullet)$ and the
   claim follows.
 %\end{enumerate}
\end{proof}

We now consider bi-Lipschitz transformations of the geometric setting.

\begin{proposition} \label{p-transform}
Let $\Lambda \subseteq \R^d$ be a bounded, open set that is a Lipschitz domain, i.e., $\Lambda$ satisfies Assumption~\ref{a-general}~\ref{a-general-1} in every point $\x \in \partial \Lambda$. Let $\Sigma$ be a closed
subset of its boundary. Assume that $\phi$ is a mapping from a
neighbourhood of $\overline {\Lambda}$ into $\R^d$ that is bi-Lipschitz. Let
us denote $\Lambda^\# \coloneqq  \phi(\Lambda)$ and $\Sigma^\# \coloneqq  \phi(\Sigma)$. Then the following holds true.
\begin{enumerate}[(a)]
\item \label{p-transform:i}For every $p \in (1, \infty)$ and every $\alpha \in (0,1)$, the mapping $\phi$
        induces a linear, topological isomorphism $\Phi f \coloneqq f \circ \phi$ acting between
        \[ \Phi \colon W_{\Sigma^\#}^{1,p}(\Lambda^\#) \to W^{1,p}_{\Sigma}(\Lambda) \quad
          \text{and} \quad C^\alpha(\Lambda^\#) \to C^\alpha(\Lambda).
        \]
        %
% \item \label{p-transform:ii}The adjoint operator $\Phi_{p'}^*$ is a linear,
%         topological isomorphism between $W_\Sigma^{-1,p}(\Lambda)$ and $W^{-1,p}_{\Sigma^\#}(\Lambda^\#)$.
      \item \label{p-transform:iii} Let $\omega$ be an essentially bounded, measurable function
        on $\Lambda $, taking its values in the set of $(d \times d)$-matrices. Then
        \[ \Phi^* \Bigl[-\nabla \cdot \omega \nabla\Bigr]\Phi = -\nabla \cdot \omega^\# \nabla
        \]
        with
        \begin{equation}\label{e-transformat1}
          \omega^\# (\mathrm{y}) \coloneqq  
         \frac{(D\phi)\bigl(\phi^{-1}(\mathrm{y})\bigr)\;
                \omega\bigl(\phi^{-1}(\mathrm{y})\bigr)\; ( D \phi )^T
                \bigl(\phi^{-1} (\mathrm{y})\bigr)}{\big| \det
                (D\phi)\bigl(\phi^{-1}({\mathrm{y}})\bigr)\big|}               \end{equation} 
        for almost all $\mathrm{y} \in \Lambda^\#$. Here, $D\phi$ denotes the Fr\'echet derivative of $\phi$ and $\det(D\phi)$ the corresponding determinant.
      \item \label{p-transform:iv} If $\omega$ is real and uniformly elliptic almost
        everywhere on $\Lambda$, then so is $\omega^\#$ on $\Lambda^\#$.
\end{enumerate}
\end{proposition}
\begin{proof}
  The proof of~\ref{p-transform:i} for the Sobolev spaces is contained in~\cite[Thm~2.10]{ggkr};
  for the H\"older spaces it is easy to verify. Part~\ref{p-transform:iii} is well known,
  see~\cite{hall} for an explicit verification,
  or~\cite[Ch.~0.8]{ausch}. Finally,~\ref{p-transform:iv} is implied by~\eqref{e-transformat1}
  and the fact that for a bi-Lipschitz function $\phi$ the derivative $D\phi$ and its inverse
  $(D \phi)^{-1}$ are essentially bounded, see~\cite[Ch~3.1]{ev/gar}.
\end{proof}

It will be very useful that the class $(A_\gamma)$ as in Assumption~\ref{a-Interface}~\ref{a-Interface:i} is preserved under bi-Lipschitz transformations, precisely: 
\begin{lemma} \label{l-erhaeltsich}
Let $\phi\colon\R^d \to \R^d$ be a bi-Lipschitz map and assume that $\Omega$ and $D$ satisfy
Assumption~\ref{a-Interface}~\ref{a-Interface:i}, so $\Omega$ is of class $(A_\gamma)$ at
$D$. Then $\phi(\Omega)$ is of class $(A_{\gamma_\phi})$ at $\phi(D)$, that is, there is a constant $\gamma_\phi>0$ such that for all $\y \in D$ and all $r \in (0,1]$:
\[ \lambda_d \bigl( B_r(\phi(\y)) \setminus \phi(\Omega) \bigr) \ge  \gamma_\phi \lambda_d \bigl( B_r(\phi(\y) \bigr).
\]

%\BR vielleicht doch $\phi$ nur in einer Umgebung der Menge definieren?? \\
%\ER
%   \renewcommand{\labelenumi}{\roman{enumi})}
% \begin{enumerate} 
% \item
% Assume that $\y \in \partial \Lambda$ is a point which satisfies Condition (\ref{e-anteil}).
%  Then there is a number $\hat {\gamma}$ and and a $\hat R >0$ such that 
% \[
% \mes \bigl (B_r(\phi(\y)) 
% \setminus \phi(\Lambda)) \ge \hat \gamma \mes \bigl (B_r(\phi(\y))\bigr )
% \]
%  holds for all $r \in ]0,\hat R[$.
% \item
% Assume that the assumption in i) is satisfied for the elements $\y$ from a certain 
% subset $\Sigma \subset \partial \Lambda$ and the constants $\gamma$ and $R$ are uniform 
% on $\Sigma$. Then the constants $\hat \gamma$ and $\hat R$ also may be taken uniform for
% all points from $\phi(\Sigma)$.
% \end{enumerate}
\end{lemma} 
\begin{proof}
For every Lebesgue measurable set $B \subseteq \R^d$ one has $\lambda_d(B) \ge \frac{1}{\ell^d}
\lambda_d(\phi^{-1}(B))$, where $\ell$ is a Lipschitz constant of $\phi^{-1}$, cf.\@
\cite[Chapter~3.3]{ev/gar}. This entails for every $\y \in D$ and all $r \in (0,1]$ that
\begin{align*}
 \lambda_d \bigl( B_r(\phi(\y)) \setminus \phi(\Omega) \bigr) &\ge \frac{1}{\ell^d}  \lambda_d \bigl(\phi^{-1} \bigl( B_r(\phi(\y)) \setminus \phi(\Omega) \bigr) \bigr) = \frac{1}{\ell^d}  \lambda_d \bigl(\phi^{-1} \bigl( B_r(\phi(\y)) \bigr) \setminus \Omega \bigr). \\
 \intertext{But $\phi^{-1}(B_r(\phi(\y)))$ contains the ball $B_\frac {r}{L}(\y)$, where $L \ge 1$ is a Lipschitz constant of $\phi$. Using this and Assumption~\ref{a-Interface}~\ref{a-Interface:i} we may continue to estimate by}
 &\ge \frac{1}{\ell^d} \lambda_d \bigl( B_{\frac rL}(\y) \setminus \Omega \bigr) \ge \frac{\gamma}{\ell^d} \lambda_d \bigl( B_{\frac rL}(\y) \bigr) = \frac{\gamma}{\ell^d L^d} \lambda_d\bigl(B_r(\phi(\y)) \bigr)
% = \tilde \gamma \lambda_d \bigl(B_r(\phi(\y)) \bigr)
\end{align*}
and we are done.
\end{proof}

As a final step in this preparatory subsection, we prepare the reflection argument in the proof of
Theorem~\ref{t-result}. For this we consider the matrix $R \coloneqq  \diag (1,1,\dots,1,-1) \in
\R^{d\times d}$ and define the bi-Lipschitz map $\phi_R(\x) = R \x$ for $\x \in \R^d$ that reflects at the plane $[x_d = 0]$.
\begin{lemma} \label{p-spiegel}
 Let $\Lambda \subseteq [x_d <0]$ be open and bounded and define $\Gamma$ as the
 (relative) interior of ${\partial \Lambda \cap [x_d=0]}$ in the plane $[x_d = 0]$. Furthermore, set $\Sigma \coloneqq  \partial \Lambda \setminus
 \Gamma$ and $\widehat \Lambda \coloneqq  \Lambda \cup \Gamma \cup \phi_R(\Lambda),$ and
 consider for $v \in W^{1,2}_\Sigma(\Lambda)$ the reflected function $\hat v$ on $\widehat \Lambda$ with
 \[ \hat v (\y)\coloneqq \begin{cases} v(\y) &\text{if } \y \in \Lambda, \\ v(R\y) \quad &\text{if } R\y \in \Lambda.
 \end{cases}
 \]
 Then the following holds:
 \begin{enumerate}[(a)] 
  \item \label{p-spiegel:i} If $v \in W^{1,2}_\Sigma(\Lambda)$, then $\hat v \in
    W^{1,2}_0(\widehat \Lambda)$.
  \item \label{p-spiegel:ii}
  Consider  $\Phi_2$ defined as in Proposition~\ref{p-transform} for $\phi = \phi_R$. Let $f \in W^{-1,2}_\Sigma(\Lambda)$ and set \[ \langle \hat f,\psi \rangle \coloneqq \langle f,\psi|_\Lambda \rangle + 
\langle \Phi_2^*f,\psi|_{\phi_R(\Lambda)}\rangle, \qquad \psi \in C_c^\infty(\widehat \Lambda).
\]
 Then  $f
\mapsto \hat f$ is continuous from $W^{-1,p}_\Sigma(\Lambda)$ to $W^{-1,p}_0(\widehat\Lambda)$
for every $p \geq 2$. 
\item \label{p-spiegel:iii}  Let $\eta \colon \Lambda \to \R^{d\times d}$. Define the reflected coefficient function $\hat {\eta} $ on $\widehat \Lambda$ by
  \[ \hat {\eta} (\y)\coloneqq \begin{cases} \eta(\y)  &\text{if } \y \in \Lambda, \\ R\,
      \eta(R\y)\, R \quad &\text{if } R\y \in \Lambda. \end{cases} 
  \] Let $v$ and $f$ as before. Then we have
  \begin{equation*}
    -\nabla \cdot \eta \nabla v = f \qquad \implies \qquad -\nabla \cdot \hat \eta \nabla \hat v =\hat f.
  \end{equation*}
\end{enumerate}
\end{lemma}
\begin{proof}
  In order to prove~\ref{p-spiegel:i}, note first that---thanks to the special geometric
  constellation---every $\psi \in C_\Sigma^\infty(\Lambda)$ can be extended by zero to the whole
  half space $H_-\coloneqq [x_d <0]$, resulting in a function in
  $W^{1,2}(H_-)$. By the density of $C_\Sigma^\infty(\Lambda)$ in $W^{1,2}_\Sigma(\Lambda)$ it
  follows that this extending procedure provides an isometry $E_0$ from
  $W^{1,2}_\Sigma(\Lambda)$ into $W^{1,2}(H_-)$.  Now let $v \in W^{1,2}_\Sigma(\Lambda)$. We
  consider $E_0 v$ and reflect this function across the boundary of $H_-$ to obtain a function
  $v_\pm \in W^{1,2}(\R^d)$ on all of $\R^d$ that satisfies
\begin{equation*}
\|v_\pm \|_{W^{1,2}(\R^d)} = 2\|E_0 v \|_{W^{1,2}(H_-)} = 2 \|v\|_{W^{1,2}_\Sigma(\Lambda)}.
\end{equation*}
This is easily verified by direct
calculations. So, summing up, the mapping 
\[
v \mapsto E_0 v \mapsto v_\pm \mapsto v_\pm|_{\widehat \Lambda} = \hat v
\]
is continuous from $W^{1,2}_\Sigma(\Lambda)$ to $W^{1,2}(\widehat \Lambda)$. It remains to show that indeed $\hat v \in W^{1,2}_0(\widehat\Lambda)$. To this end, let
$(v_k) \subset C^\infty_\Sigma(\Lambda)$ be an approximating sequence for $v$ in
$W^{1,2}_\Sigma(\Lambda)$. Note that it is clear that $(v_k)_\pm|_{\widehat \Lambda}$
approximates $\hat v$ in $W^{1,2}(\widehat \Lambda)$ and the supports of
$(v_k)_\pm|_{\widehat \Lambda}$ have a positive distance to $\partial \widehat \Lambda$, but
the functions are not smooth any more in general. But this can be rectified by mollifying each
$(v_k)_\pm|_{\widehat \Lambda}$ with a suitable regularizing kernel such that the resulting
smooth functions' supports still have a positive distance to $\partial \widehat \Lambda$, and it
is easily shown that these functions still approximate $\hat v$ in $W^{1,2}(\widehat
\Lambda)$, so $\hat v \in W^{1,2}_0(\widehat\Lambda)$.

The proof of~\ref{p-spiegel:ii} and~\ref{p-spiegel:iii} is concluded from a straightforward
calculation and application
of the definitions of the operators $-\nabla \cdot  \mu \nabla$ and $-\nabla \cdot \hat \mu
\nabla$ together with Proposition~\ref{p-transform}.
\end{proof}

\begin{rem}
  \label{rem:uniform-procedure}
  From the proofs of the foregoing framework for localization, transformation and reflection it
  is easily seen that each step preserves uniform bounds in the data of an elliptic equation, that
  is, the underlying geometry, the right-hand side, and the coefficient function. In this sense,
  whenever a result on elliptic regularity on the localized, transformed or reflected level
  yields a uniform estimate on the solution in the aforementioned data, this uniform estimate
  carries over to the original situation immediately. Of course, this is exactly the case for
  our main tool, Proposition~\ref{p-lady}.
\end{rem}

%%%%%%%%%%%%%%%%%%%%%%%%%%%%%%%%%%%%%%%%%%%%%%%%%%%%%%%%%%%%%%%%%%%%%%%%%%%%%
\subsection {Proof of Theorem~\ref{t-result}}\label{sec:proof-theorem-reft}
We now start the proof of the H\"older continuity following the program sketched in the
preceding subsection, cf.\@ page~\pageref{p-Ablaufplan}. According to the hypotheses of
Theorem~\ref{t-result}, from now on we suppose that $\Omega$ and $D$ satisfy the
Assumptions~\ref{a-general}~\ref{a-general-1} and, if $d=4$, also~\ref{a-general-2}, as well as (always) Assumption~\ref{a-Interface}.

In order to start the localisation procedure, we fix some notation. For the Neumann boundary part we use again the shorthand $N = \partial \Omega \setminus D$. Now, based on Assumption~\ref{a-general}~\ref{a-general-1}, choose for every $\x \in \overline{N}$ an associated open neighbourhood $V_\x$ and let $\{ V_{\x_1},\ldots, V_{\x_m} \}$ be a finite subcovering of $\overline{N}$. 

Furthermore, choose a bounded open neighbourhood $W$ of $\overline \Omega$ and put $U_0
\coloneqq  W \setminus \overline{N}$. Then $U_0$ is open and one has 
\[ U_0 \cap \Omega =\Omega \quad\text { and } \quad U_0 \cap \overline{N} =\emptyset.
\]
The system $\mathcal U \coloneqq  \{U_0, V_{\x_1}, V_{\x_2}, \dots V_{\x_m}\}$ forms an open
covering of $\overline \Omega$. Moreover, all sets in $\mathcal U$ give rise to extension
domains; this will come in handy in view of Lemma~\ref{e-abschfbullet}:
\begin{lemma} \label{l-extendprojected}
Let $U \in \mathcal U$ and put $\Omega_\bullet\coloneqq {\Omega \cap U}$ and $D_\bullet
\coloneqq  \partial \Omega_\bullet \setminus N$. Then for all  $r \in (1,\infty)$ the space $W^{1,r}_{D_\bullet}(\Omega_\bullet)$ admits again the continuation property, i.e., there is a continuous extension operator $\mathfrak E_U \colon W^{1,r}_{D_\bullet}(\Omega_\bullet) \to W^{1,r}(\R^d)$.
\end{lemma}
\begin{proof}
In the case $U = U_0$ one has $D_\bullet = \partial \Omega_\bullet$ by construction. Thus, $W^{1,r}_{D_\bullet}(\Omega_\bullet) = W^{1,r}_0(\Omega_\bullet)$ and the trivial extension by zero does the trick even without any condition on the boundary. If $U = V_{\x_j}$, then $\Omega_\bullet = {\Omega \cap V_{\x_j}}$ is mapped onto the lower half cube $\{ \x \in (-1,1)^d \colon x_d < 0\}$ by the bi-Lipschitz map $\phi_{\x_j}$ that is defined on a neighbourhood of 
$\overline {\Omega_\bullet}$, cf.~Assumption~\ref{a-general}~\ref{a-general-1}. The lower half cube is a Lipschitz domain. Thus, $\Omega_\bullet$ is also a Lipschitz domain and there is even an extension operator from $W^{1,r}(\Omega_\bullet)$ into $W^{1,r}(\R^d)$ thanks to Proposition~\ref{p-2}. (Choose there $\Omega = \Omega_\bullet$ and $D = \emptyset$.)
\end{proof}

Corresponding to the open covering $\mathcal U$ of $\overline{\Omega}$ we choose a smooth partition of
unity $\{\eta_0,\eta_1,\ldots,\eta_m\} \subset C_c^\infty(\R^d)$ such that $\supp(\eta_0) \subseteq U_0$ and
$\supp(\eta_j) \subseteq V_{\x_j}$ for $j \in \{1,\ldots , m\}$.

Let from now on $q>d$ be fixed, let $f \in W^{-1,q}_D(\Omega)$, and let $v \in
W^{1,2}_D(\Omega)$ be the solution to~\eqref{e-euq}, so $(\cA + 1)v = f$. We write $v = \sum_{j=0}^m \eta_j v$ and aim to show the H\"older continuity of every function $\eta_j v$ seperately. The easiest case is $j=0$:
\begin{lemma}\label{lem:localize-interior}
 There exists an $\alpha_0 > 0$ independent of $f$ such that $\eta_0 v \in C^{\alpha_0}(\Omega)$ and the estimate
 \[ \|\eta_0 v\|_{C^{\alpha_0}(\Omega)} \le c \|f\|_{W^{-1,q}_D(\Omega)}
 \]
 holds true. The constant $c$ depends only on geometry and the $L^\infty(\Omega)$-bound and ellipticity constant of the given coefficient
  function $\mu$.
\end{lemma}
\begin{proof}
Since $\overline N$ does not intersect $U_0$, the function $\eta_0 v$ belongs to
$W^{1,2}_0(\Omega)$, cf.\@ Lemma~\ref{l-restr/ext}. Moreover, by Lemma~\ref{l-project} there is
a $p >d$ and $f_0 \in W^{-1,p}(\Omega)$ such that the function $\eta_0v$ satisfies the equation
\begin{equation} \label{e-prosubgeb}
-\nabla \cdot \mu \nabla (\eta_0v) =f_0.
\end{equation}
Since we are now in the setting of a pure Dirichlet problem and have
Assumption~\ref{a-Interface}~\ref{a-Interface:i} at our disposal, we can apply
Proposition~\ref{p-lady}; note also Remark~\ref{r-interpret}. This yields that the solution
$\eta_0 v$ of \eqref{e-prosubgeb} is H\"olderian of some degree $\alpha_0$ with the estimate
\begin{align*}
 \|\eta_0 v\|_{C^{\alpha_0}(\Omega)} &\le c \|f_0\|_{W^{-1,p}(\Omega)}.
 \intertext{Finally, combining Lemma~\ref{l-extendprojected} and
Lemma~\ref{e-abschfbullet}, we conclude that}
\|\eta_0 v\|_{C^{\alpha_0}(\Omega)} &\le  c \|f\|_{W^{-1,q}_D(\Omega)},
\end{align*}
where $\alpha_0$ does not depend on $f$. For the uniformity claim, see Remark~\ref{rem:uniform-procedure}.
\end{proof}
We turn to the H\"older continuity of the functions $\eta_j v$ for $j \in \{1,\ldots,m\}$. For
these, there will be a part of the Neumann boundary $N$ present. To make do with this,
we transform the localized problems via the diffeomorphisms $\phi_{\x_j}$ to the model
constellation on the unit cube as in Assumption~\ref{a-general}~\ref{a-general-1}, which enables
us to use a
reflection argument to end up in a situation with a pure Dirichlet boundary condition. Then we can conclude by Proposition~\ref{p-lady}. For this we introduce the notation $Q \coloneqq  (-1,1)^d$ for the unit cube, $Q_- \coloneqq  \{ \x \in Q \colon x_d < 0 \}$ for its lower half and $P \coloneqq  \{\x \in Q \colon x_d = 0\}$ for its midplate.

Due to Lemma~\ref{l-project}, there is $p>d$ such that each of the functions $\eta_j v$, $j = 1, \dots, m$ satisfies an equation like 
\begin{equation} \label{e-prosubgeb1}
-\nabla \cdot \mu \nabla (\eta_jv) =
f_j \in W_{D_j}^{-1,p}(\Omega \cap V_{\x_j}),
\end{equation}
with $D_j = \partial (\Omega \cap V_{\x_j}) \setminus N$. Note that the right hand sides $f_j$ continuously depend on $f$, see Lemma~\ref{e-abschfbullet}. According to Proposition~\ref{p-transform}, one may transform these equations under the bi-Lipschitz diffeomorphisms $\phi_{\x_j}$ and pass to the equation 
\begin{equation} \label{e-prosubgeb2}
 -\nabla \cdot \mu_j^\# \nabla w_j = g_j \in W_{\Sigma_j}^{-1,p}(Q_-),
\end{equation}
where $\Sigma_j =\phi_{\x_j}(D_j) \subseteq \partial Q_-$ is the transformed Dirichlet part,
$w_j \in W^{1,2}_{\Sigma_j}(Q_-)$ is the transformed version of the function $\eta_j v|_{\Omega \cap V_{\x_j}}$ and $g_j$ is
the transformation of $f_j$. Note that the whole 'lower mantle' boundary $\partial Q_- \setminus P$ belongs to $\Sigma_j$, since $\phi_{\x_j}(N \cap V_{\x_j}) \subseteq P$.

From now on we  distinguish whether $\x_j \in N$ or $\x_j \in D \cap \overline{N}$, starting
with the former.
\begin{lemma}\label{lem:localize-N}
 Let $j \in \{1,2, \dots, m\}$ with $\x_j \in N$. Then there is some $\alpha_j > 0$ independent of $f$ such that $\eta_j v \in C^{\alpha_j}(\Omega)$ and we have
 \[ \|\eta_j v\|_{C^{\alpha_j}(\Omega)} \le c \|f\|_{W^{-1,q}_D(\Omega)}.
 \]
 The constant $c$ depends only on geometry and on the $L^\infty(\Omega)$-bound and ellipticity constant of the given coefficient
  function $\mu$.
\end{lemma}
\begin{proof}
  Thanks to Remark~\ref{r-shrink} we can assume that
  $\Sigma_j = \partial Q_- \setminus P$. Thus, exploiting Lemma~\ref{p-spiegel} with
  $\Lambda = Q_-$ and $\Gamma = P$, the symmetrically reflected function $\widehat{w_j}$ belongs
  to the space $W^{1,2}_0(Q)$ and obeys an elliptic equation on the cube $Q$ with the right hand
  side  $\widehat{g_j} \in W^{-1,p}(Q)$. The cube $Q$ is obviously convex and satisfies the
  regularity condition in Proposition~\ref{p-lady} with $\gamma = 1/2$. Thus, said
  Proposition~\ref{p-lady} applies and gives us H\"older continuity of $\widehat{w_j}$ of
  degree, say, $\alpha_j$, with an
  estimate in $\widehat{g_j} \in W^{-1,p}(Q)$. By Proposition~\ref{p-transform} and
  Lemma~\ref{p-spiegel} we then have
  \begin{multline*}
    \|\eta_j v\|_{C^{\alpha_j}(\Omega \cap V_{\x_j})} \leq c\|w_j\|_{C^{\alpha_j}(Q_-)} \leq
    c\|\widehat{w_j}\|_{C^{\alpha_j}(Q)} \\ \leq c\|\widehat{g_j}\|_{W^{-1,p}(Q)} \leq c\|g_j\|_{W_{\Sigma_j}^{-1,p}(Q_-)} \leq c\|f_j\|_{W_{D_j}^{-1,p}(\Omega \cap V_{\x_j})} \leq c\|f\|_{W^{-1,q}_D(\Omega)}.
  \end{multline*}
  %
  % This inherits to the functions $w_j$
  % (Lemma~\ref{p-spiegel}) and $\eta_j v|_{\Omega \cap V_{\x_j}}$
  % (Proposition~\ref{p-transform}), with an estimate ultimately in $f \in W^{-1,q}_D(\Omega)$.
  Since the
  support of $\eta_j$ has a positive distance to $\Omega \setminus V_{\x_j}$, the $\alpha_j$-H\"older
  continuity and norm estimate is preserved for $\eta_j v$ on the whole $\Omega$. For the
  uniformity claim, see again Remark~\ref{rem:uniform-procedure}.
\end{proof}
It remains to treat the patches with $\x_j \in D \cap \overline{N}$ and it is here that
Assumption~\ref{a-Interface}~\ref{a-Interface:ii} comes into play. In order to reformulate this
condition in our current notation, for some set $M \subseteq \partial Q_-$, we denote its
relative boundary inside $\partial Q_-$ by $\mathrm{bd}_{\partial Q_-}(M)$ and inside $P$ by
$\mathrm{bd}_P(M)$. Then Assumption~\ref{a-Interface}~\ref{a-Interface:ii} reads as follows: There are two constants $c_0 \in
(0,1)$ and $c_1 > 0$, such that for all $(\y,0) \in  \mathrm{bd}_P(\Sigma_j)$ and all $s \in
(0,1]$ we have
\begin{equation} \label{e-1}
 \lambda_{d-1} \bigl( \bigl\{ \z \in \lowdim{B}_s(\y) \colon \dist (\z, P \setminus \Sigma_j) > c_0 s \bigr\} \bigr) \ge c_1 s^{d-1}.
\end{equation}
Later on it will be convenient to have this condition not only for the points in the interface $\mathrm{bd}_P(\Sigma_j)$, but for all points of $\Sigma_j$ inside $P$. It is an interesting fact that this comes for free, once we suppose it on the interface. This will be elaborated in the next two lemmas.
\begin{lemma} \label{l-extendfact}
Condition \eqref{e-1} carries over to all points $(\y,0)
\in \mathrm{bd}_{\partial Q_-}(\Sigma_j)$ with possibly different constants $c_0, c_1 >0$.
\end{lemma} 
\begin{proof}
Since $\partial Q_- \setminus P \subseteq \Sigma_j$, we have the inclusion
\[ \mathrm{bd}_{\partial Q_-}(\Sigma_j) = \mathrm{bd}_{\partial Q_-}(\Sigma_j \cap P) \subseteq \mathrm{bd}_P(\Sigma_j) \cup \mathrm{bd}_{\partial Q_-}(P).
\]
For $(\y,0) \in \mathrm{bd}_{\partial Q_-}(P)$ we estimate
\begin{align*}
\lambda_{d-1} \Bigl(\Bigl\{ \z \in \lowdim{B}_s(\y) \colon \dist(\z, P \setminus \Sigma_j) >  \frac {s}{2} \Bigr\} \Bigr) &\ge \lambda_{d-1} \Bigl( \Bigl\{ \z \in \lowdim{B}_s(\y) \colon \dist(\z, P) >  \frac {s}{2} \Bigr\} \Bigr) \\
&\ge \frac {\omega_{d-1}}{2^{d-1}} s^{d-1}.
\end{align*}
So, \eqref{e-1} is true for all points $(\y,0)$ in $\mathrm{bd}_{Q_-}(P)$ and it is true for all $(\y,0)$ in $\mathrm{bd}_P(\Sigma_j)$ by hypotheses, with possibly different constants $c_0$ and $c_1$. In order to conclude, it suffices to observe the following: If for a point $\y$ and  a number $s>0$ the inequality~\eqref{e-1} holds, then this remains true if the constants $c_0, c_1$ are replaced
by smaller ones.
\end{proof}
\begin{lemma} \label{l-5.7}
 We have for all $(\y,0) \in \Sigma_j \cap P$ and all $s \in (0,1]$
 \[ \lambda_{d-1} \bigl(\bigl\{ \z \in \lowdim{B}_s(\y) \colon \dist(\z, P \setminus \Sigma_j) > {\hat c}_0 s \bigr\} \bigr) \ge {\hat c}_1 s^{d-1},
 \]
 for $\hat c_0 \coloneqq  \min\{\frac {1}{4}, \frac {c_0}{2}\}$ and $\hat c_1 \coloneqq 
 \min \{\frac { \omega_{d-1}}{4^{d-1}}, \frac {c_1}{2^{d-1}} \}$, where $c_0$ and $c_1$ are from Lemma~\ref{l-extendfact}.
\end{lemma}

\begin{proof}
For all $(\y,0) \in \mathrm{bd}_P(\Sigma_j)$ the assertion is true by Assumption~\ref{a-Interface}~\ref{a-Interface:ii} and, using again the observation made in the end of the proof of Lemma~\ref{l-extendfact}, it suffices to treat the case where $(\y,0)$ is a relatively inner point of $\Sigma_j$ in $P$. Since $\overline{P \setminus \Sigma_j}$ is compact, we then have
 \[ \eps \coloneqq  \dist( \y, P \setminus \Sigma_j) = \dist (\y, \overline{P \setminus \Sigma_j}) > 0.
 \]
 We distinguish three cases:

 \paragraph{\textbf{First case, $0 < s \le \eps/2$:}} In this case one finds
   \[ \bigl\{ \z \in \lowdim{B}_{s}(\y) \colon \dist ( \z, P \setminus \Sigma_j)  >  s \bigr\}=\lowdim{B}_{s}(\y),
   \]
   so
   \[ \lambda_{d-1} \bigl( \bigl\{ \z \in \lowdim{B}_{s}(\y) \colon \dist (\z, P \setminus \Sigma_j) > s \bigr\} \bigr) = \lambda_{d-1} \bigl( \lowdim{B}_{s} (\y) \bigr) =  {\omega_{d-1}} s^{d-1}.
   \]
  \paragraph{\textbf{Second case, $\eps/2 < s \le 2 \eps$:}} Since $s/4 \le \eps/2$, we infer from the first case 
        \begin{align*}
          \lambda_{d-1} \Bigl( \Bigl\{ \z \in \lowdim{B}_{s}(\y) \colon \dist (\z, P \setminus \Sigma_j) > \frac {s}{4} \Bigr) \Bigr\} &\ge \lambda_{d-1} \Bigl( \Bigl\{ \z \in \lowdim{B}_{ \frac {s}{4}}(\y) \colon \dist (\z, P \setminus \Sigma_j) > \frac {s}{4} \Bigr\} \Bigr) \\
          &\ge \omega_{d-1} \frac {s^{d-1}}{4^{d-1}}.
        \end{align*}
  \paragraph{\textbf{Third case, $2 \eps < s \le 1$:}}
        From the fact that $\overline {P \setminus \Sigma_j}$ is compact, we not only get that
        $\eps > 0$, but we also obtain the existence of a point $(\y^*,0) \in
        \mathrm{bd}_{\partial Q_-}(\Sigma_j)$ with $\| \y - \y^*\|_{\R^{d-1}} = \eps$. Since $\lowdim{B}_{s - \eps}(\y^*) \subseteq \lowdim{B}_{s}(\y)$, this yields
        \[ \lambda_{d-1} \Bigl( \Bigl\{ \z \in \lowdim{B}_{s}(\y) \colon \dist (\z, P \setminus \Sigma_j) > \frac {c_0}{2}s \Bigr\} \Bigr) \ge \lambda_{d-1} \Bigl( \Bigl\{ \z \in \lowdim{B}_{s-\eps}(\y^*) \colon \dist (\z, P \setminus \Sigma_j) >  \frac {c_0}{2}{s} \Bigr\} \Bigr).
        \]
        The condition $2\eps < s$ implies $\frac{c_0}{2} s < c_0 (s - \eps)$. 
        Using this and Lemma~\ref{l-extendfact}, we continue to estimate
        \[
          \ge  \lambda_{d-1} \bigl( \bigl\{ \z \in \lowdim{B}_{s-\eps}(\y^*) \colon \dist (\z, P \setminus \Sigma_j) > c_0 (s - \eps \bigr\} \bigr) \ge  c_1 (s - \eps)^{d-1} \ge  \frac{c_1}{2^{d-1}} s^{d-1}.
        \]
        Invoking once more the observation from the end of the proof of Lemma~\ref{l-extendfact}, we deduce the claim.
\end{proof}

Let, in all what follows, $\hat c_0, \hat c_1$ be the constants from Lemma~\ref{l-5.7}. Also, we will
often use the decomposition $\R^d \ni \x = (\bar\x,x_d) \in \R^{d-1} \times \R$.

For $t \in \R$, we define the mapping $\psi_t \colon \R^d \to \R^d$ by
\begin{equation}
 \psi_t(\x) = \psi_t\bigl( (\bar\x,x_d) \bigr) \coloneqq 
\bigl(\bar\x, x_d - t \dist(\bar\x, P \setminus \Sigma_j)\bigr).\label{eq:def-psi}
\end{equation}
Later on we will transform our problem again under the mapping $\psi_t$ for a suitably chosen value of $t$ and afterwards reflect it in correspondence with Lemma~\ref{p-spiegel}. In order to justify this transformation, we first prove a little lemma.
\begin{lemma} \label{t-hilfs} Consider $\psi_t$ be as in~\eqref{eq:def-psi}. Then the following holds true:
 \begin{enumerate}[(a)] 
  \item\label{t-hilfs:i} The function $\R^d \ni \x = (\bar\x,x_d) \mapsto \dist(\bar\x, P \setminus
    \Sigma_j)$ is a Lip\-schitz contraction. 
  \item\label{t-hilfs:ii} For every $t \in \R$, the function $\psi_t$ is Lip\-schitz continuous and bijective with
    inverse $\psi_{-t}$. In particular, the inverse is also Lip\-schitz continuous.
  \item\label{t-hilfs:iii} For every $t \in \R$, the function $\psi_t$ is volume preserving.
\end{enumerate} 
\end{lemma}
\begin{proof}
The function under consideration in~\ref{t-hilfs:i} is the concatenation of the projection
 $\R^d \ni \x \mapsto (\bar\x,0)$ onto $[x_d =0]$ and the restriction of the function $\R^d \ni
 \x \mapsto \dist(\x, P \setminus \Sigma_j)$ to $\R^{d-1} \times \{0\}$. Both of these functions
 are Lipschitz continuous contractions, thus so is the considered concatenation.
 
For~\ref{t-hilfs:ii}, the first assertion follows from~\ref{t-hilfs:i}, and the second is easy
to verify.

Finally, it is clear that the determinant of the Jacobian of $\psi_t$  (cf.~\cite[Chapter~3.2.2]{ev/gar}) is identically $1$ a.e., thus the assertion~\ref{t-hilfs:iii} follows from \cite[Chapter~3.3.3 Theorem~2]{ev/gar}.
\end{proof}
In the following we choose $t = \frac{3}{\hat c_0}$ and abbreviate $\psi \coloneqq  \psi_{3/\hat c_0}$. We transform \eqref{e-prosubgeb2} under $\psi$ to a problem
\begin{equation} \label{e-prosubgeb3}
-\nabla \cdot \omega \nabla w = h \in W_{\Sigma_\Delta}^{-1,p}(Q_\Delta),
\end{equation}
where the resulting domain is $Q_\Delta\coloneqq \psi(Q_-)$ and the new Dirichlet boundary part
is $\Sigma_\Delta \coloneqq \psi(\Sigma_j)$. We suppress the dependence on $j$ here, so $w$ is
the transformation of $w_j$ by slight abuse of notation, and $h$ is the transformed $g_j$. Furthermore, the resulting coefficient function $\omega$ is again real, elliptic
and bounded thanks to Proposition~\ref{p-transform}.

The
crucial effect of the transformation $\psi$ is that the new Neumann boundary part
$N_\Delta \coloneqq \partial Q_\Delta \setminus \Sigma_\Delta$ is identical to the old Neumann
part $P \setminus \Sigma_j$ and that
$\overline{N_\Delta} = \overline{P \setminus \Sigma_j} = \partial Q_\Delta \cap
\overline{P}$. In particular, $\partial Q_\Delta \cap P = \partial Q_\Delta \cap [x_d = 0]$ consists of Neumann boundary
only. Thus, the geometry of the problem~\eqref{e-prosubgeb3} is now exactly of the shape needed to reflect the problem across the plane $[x_d=0]$, according to Lemma~\ref{p-spiegel}. We end up with the domain
\[ \Lambda \coloneqq  Q_\Delta \cup N_\Delta \cup \bigl\{\z = (\bar \z,z_d) \in \R^d \colon (\bar
  \z,-z_d) \in Q_\Delta \bigr\}, 
\]
while the new coefficient function $\hat\omega$ is again real, bounded and elliptic and the
resulting right hand side $\hat h$ belongs to the space $W^{-1,p}(\Lambda)$ with
$p>d$. Lemma~\ref{p-spiegel} already tells us that the solution $\hat w$ of the equation on
$\Lambda$ belongs to $W^{1,2}_0(\Lambda)$. Thus, in order to infer H\"older continuity for $\hat
w$ by Proposition~\ref{p-lady}, the only thing that is left to verify is that our final geometry
satisfies Assumption~\ref{a-Interface}~\ref{a-Interface:i}. This will be the main part of the proof.
\begin{lemma} \label{l-afat}
The domain $\Lambda$ is of class $(A_\gamma)$ for some $\gamma \in (0,1)$.
\end{lemma}
\begin{proof}
 The boundary of $\Lambda$ is the union of the sets $\psi(\Sigma_j \cap P)$ and $\psi(\partial
 Q_- \setminus P)$ and their reflected  counterparts. We show the assertion for the points from
 $\psi(\Sigma_j \cap P)$ and from $\psi(\partial Q_- \setminus P)$, the proof for points from the reflected parts is then analogous.

 Let $r \in (0,1]$ and assume $\y=(\bar\y,y_d) \in \psi(\Sigma_j\cap P)$. Then $\y$ is necessarily of the
 form $\bigl(\bar\y, -3/\hat c_0 \cdot \dist (\bar\y, P \setminus \Sigma_j)\bigr)$ with $(\bar\y,0) \in
 P$. Now let first $r < 3/\hat c_0 \cdot \dist(\bar\y, P \setminus \Sigma_j)$. Then the ball
 $B_r(\y)$ lies completely in the half space $[x_d < 0]$. This gives 
 \[ B_r(\y) \setminus \Lambda = B_r(\y) \setminus Q_\Delta
  = B_r(\y) \setminus \psi(Q_-).
 \]
 Applying the volume-preserving map $\psi^{-1}$, cf. Lemma~\ref{t-hilfs}, and using that $\psi^{-1}(\y)
 \in P$, one deduces the inequality
 \begin{align*}
  \lambda_d \bigl( B_r(\y) \setminus \Lambda \bigr) &= \lambda_d \bigl( \psi^{-1} ( B_r(\y) ) \setminus Q_- \bigr) \\
  &\ge \lambda_d \bigl( \psi^{-1} ( B_r(\y) ) \cap [x_d>0]\bigr ) = \frac 12 \lambda_d \bigl( \psi^{-1} (B_r(\y)) \bigr). \\
  \intertext{Since $\psi^{-1}$ is Lipschitz continuous by Lemma~\ref{t-hilfs}, the set $\psi^{-1} \bigl( B_r(\y) \bigr)$ contains the ball $B_{\ell r} \bigl( (\bar\y,0) \bigr)$, where $\ell$ is the Lipschitz constant of $\psi^{-1}$. Thus, we can continue to estimate}
  &\ge \frac 12 \ell^d \omega_d r^d = \frac 12 \ell^d \lambda_d \bigl( B_r(\y) \bigr).
 \end{align*}
 Now we consider the second case $r \ge 3/\hat c_0 \cdot \dist(\y, P \setminus \Sigma_j)$. Let
 \[ B_r^-(\y) \coloneqq  B_r(\y) \cap \Bigl\{\z \in \R^d \colon z_d \le -3/\hat c_0 \cdot \dist(\bar\y, P \setminus \Sigma_j) \Bigr\}.
 \]
 Since $y_d = -3/\hat c_0 \cdot \dist(\bar\y, P \setminus \Sigma_j)$, this is exactly the 'lower' half of $B_r(\y)$.
 By construction of $\Lambda$, one has
 \begin{align*}
  B_r(\y) \setminus \Lambda &\supseteq B^-_r(\y) \setminus \Lambda =
  B^-_r(\y) \setminus Q_\Delta.
\end{align*} 
Due to the choice of $\psi$, we have $Q_\Delta \subseteq \bigl\{
(\bar\z,z_d) \in \R^d\colon   z_d \leq -3/\hat c_0 \cdot \dist(\bar\z, P
\setminus \Sigma_j)\}$. Thus, we may continue
\begin{align*}
B_r(\y) \setminus \Lambda  &\supseteq B^-_r(\y) \setminus \Bigl\{ (\bar\z,z_d) \in \R^d\colon  z_d \leq - \frac{3}{\hat c_0} \dist (\bar\z, P \setminus \Sigma_j) \Bigr\} 
  \\
  &= B^-_r(\y) \cap \Bigl\{(\bar\z,z_d) \in \R^d\colon  \frac{3}{\hat c_0} \dist (\bar\z, P \setminus \Sigma_j) > -z_d \Bigr\}.
\end{align*}
We aim to parametrize the last set by layers along the $z_d$-direction. To this end, for $s \in [0,r]$
we denote by $H_s$ the hyperplane
$\bigl\{ (\bar\z,z_d) \in \R^d\colon z_d = -3/\hat c_0 \cdot \dist(\bar\y, P \setminus \Sigma_j)
- s \bigr\}$. Then we obtain
\begin{align*}
 B_r(\y) \setminus \Lambda &\supseteq B^-_r(\y) \cap \Bigl( \bigcup_{s \in [0,r]} H_s \Bigr) \cap \Bigl\{(\bar\z,z_d) \in \R^d\colon  \frac{3}{\hat c_0} \dist(\bar\z, P \setminus \Sigma_j) > -z_d \Bigr\} \\
  &=\bigcup_{s \in [0,r]} \bigl( B^-_r(\y) \cap H_s \bigr ) \cap \Bigl\{ (\bar\z,z_d) \in \R^d\colon  \frac{3}{\hat c_0} \dist(\bar\z, P \setminus \Sigma_j) > -z_d \Bigr\} \\
 %  &= \bigcup_{s \in [0,r]} \Bigl\{ (\bar\z,z_d) \in \R^d\colon  z_d =
 %  - \frac{3}{\hat c_0} \dist(\bar\y, P \setminus \Sigma_j) - s,~\bar z \in
 %  \lowdim{B}_{\sqrt {r^2-s^2}}(\bar\y), \\
 % & \qquad \qquad \qquad \qquad\qquad
 % \frac{3}{\hat c_0} \dist(\bar\z, P \setminus \Sigma_j) >
 % \frac{3}{\hat c_0}\dist(\bar\y, P \setminus \Sigma_j) + s \Bigr\} \\
 &\eqqcolon \bigcup_{s \in [0,r]} G_s
\end{align*}
with
\begin{multline*}
  G_s \coloneqq \Bigl\{ (\bar\z,z_d) \in \R^d\colon  z_d =
  - \frac{3}{\hat c_0} \dist(\bar\y, P \setminus \Sigma_j) - s,~\bar z \in
  \lowdim{B}_{\sqrt {r^2-s^2}}(\bar\y), \\
% & \qquad \qquad \qquad \qquad\qquad
 \frac{3}{\hat c_0} \dist(\bar\z, P \setminus \Sigma_j) >
 \frac{3}{\hat c_0}\dist(\bar\y, P \setminus \Sigma_j) + s \Bigr\}.
\end{multline*}
We now note the representation $G_s = \lowdim{G}_s \times \bigl\{- \frac{3}{\hat
  c_0} \dist(\bar\y, P \setminus \Sigma_j) - s\bigr\}$ with
\begin{equation*}
  \lowdim{G}_s =  \lowdim{B}_{\sqrt
    {r^2-s^2}}(\bar\y) \cap \Bigl\{ \bar\z \in \R^{d-1} \colon
  \frac{3}{\hat c_0} \dist(\bar\z, P \setminus \Sigma_j) >
  \frac{3}{\hat c_0      }\dist(\bar\y, P \setminus \Sigma_j) + s
  \Bigr\}.
\end{equation*}
Thus, applying Cavalieri's principle,
\begin{equation*}
  \lambda_d \bigl( B_r(\y) \setminus \Lambda \bigr) \geq  \int _0^r
  \lambda_{d-1}(G_s) \dd s = \int _0^r
  \lambda_{d-1}(\lowdim{G}_s) \dd s \geq \int_\frac{r}{2}^{\frac{r}{\sqrt {2}}}  \lambda_{d-1}(\lowdim{G}_s) \dd s.
\end{equation*}
For $s \in [0,\frac{r}{\sqrt {2}}]$ we have
    $\lowdim{B}_{s}(\bar\y) \subseteq \lowdim{B}_{\sqrt {r^2-s^2}}
    (\bar\y)$. On the other hand, for $s \ge \frac {r}{2}$ the
    supposition $r \ge 3/\hat c_0 \cdot \dist(\bar\y, P \setminus
    \Sigma_j)$ yields $3s \ge r + s \ge 3/\hat c_0 \cdot \dist(\bar\y,
    P \setminus \Sigma_j) + s$. So, for $\frac{r}2 \leq s \leq \frac{r}{\sqrt{2}}$,
    \begin{equation*}
      \lowdim{G}_s \supseteq  \lowdim{B}_s(\bar\y) \cap \Bigl\{ \bar\z \in \R^{d-1} \colon \frac{3}{\hat c_0} \dist(\bar\z, P \setminus \Sigma_j) > 3s \Bigr\},
    \end{equation*}
    and using Lemma~\ref{l-5.7} we can
    continue to estimate:
    \begin{align*}
\lambda_d \bigl( B_r(\y) \setminus \Lambda \bigr)& \geq
      \int_{\frac r2}^{\frac{r}{\sqrt{2}}} \lambda_{d-1} \Bigl( \lowdim{B}_s(\bar\y) \cap \Bigl\{ \bar\z \in \R^{d-1} \colon \frac{3}{\hat c_0} \dist(\bar\z, P \setminus \Sigma_j) > 3s \Bigr\} \Bigr) \dd s\\
  & = \int_{\frac r2}^{\frac{r}{\sqrt{2}}} \lambda_{d-1} \bigl( \bigl\{ \bar\z \in \lowdim{B}_s(\bar\y) \colon \dist(\bar\z, P \setminus \Sigma_j) > \hat c_0 s \bigr\} \bigr) \dd s \\
  & \geq \hat c_1 \int_{\frac r2}^{\frac{r}{\sqrt{2}}} s^{d-1} \dd s =
  \frac{\hat c_1}{d} \Bigl[ \Bigl( \frac {1}{2} \Bigr)^\frac{d}{2} -
  \Bigl( \frac{1}{2} \Bigr)^d \Bigr] r^d = \frac{\hat c_1(2^{\frac d2}
    - 1)}{d \omega_d 2^d} \lambda_d\bigl(B_r(\y)\bigr).
\end{align*}
This was the claim for $\y \in \Sigma_j \cap P$.

It remains to discuss the points $\y \in \psi(\partial Q_- \setminus P)$. Clearly,
$\Lambda$ is contained in the 'strip' $(-1,1)^{d-1} \times \R$, and $\psi$ maps the lateral faces $M\coloneqq
\{-1,1\}^{d-1} \times [-1,0]$ of $Q_-$ into $\{-1,1\}^{d-1} \times (-\infty, 0]$ which are
exactly (the 'lower' half
of) the faces of $(-1,1)^{d-1} \times \R$. Thus, for $\y
\in \psi(M)$, the set $B_r(\y) \setminus \Lambda$ contains at least half of the ball $B_r(\y)$
and we have
\begin{align*}
 \lambda_d \bigl( B_r(\y) \setminus \Lambda \bigr) \ge \lambda_d \bigl( B_r(\y) \setminus
 ((-1,1)^d \times \R)\bigr) \ge  \frac 12 \lambda_d(B_r(\y)).
\end{align*}
The only case left is $\y \in \psi \bigl( \{-1\} \times (-1,1)^{d-1} \bigr)$, i.e., $\y$ is in
the image of the 'bottom' of the half cube. Then the ball $B_r(\y)$ lies completely inside the
'lower' halfplane $[z_d = 0]$ for all $r \in (0,1]$. Thus, since $\psi$ was volume-preserving,
\[ B_r(\y) \setminus \Lambda = B_r(\y) \setminus \psi(Q_-) = B_r \bigl( \psi(\psi^{-1}(\y)) \bigr) \setminus \psi(Q_-).
\]
By Lemma~\ref{l-erhaeltsich} we get the desired estimate once we can prove it for the
untransformed geometry $B_r(\psi^{-1}(\y)) \setminus Q_-$ where $\psi^{-1}(\y)$ is in the bottom
face of the unit cube. But this is straightforward since $Q_-$ is convex.
\end{proof}
With Lemma~\ref{l-afat} at hand, we complete the proof of Theorem~\ref{t-result} easily with the
pendant to Lemma~\ref{lem:localize-N}; its proof is completely analogous to the one of
Lemma~\ref{lem:localize-N} up to the additional transformation $\psi$.
\begin{lemma}\label{lem:localize-interface}
 Let $j \in \{1,2, \dots, m\}$ with $\x_j \in D\cap \overline{N}$. Then there is some $\alpha_j > 0$ independent of $f$ such that $\eta_j v \in C^{\alpha_j}(\Omega)$ and we have
 \[ \|\eta_j v\|_{C^{\alpha_j}(\Omega)} \le c \|f\|_{W^{-1,q}_D(\Omega)}.
 \]
 The constant $c$ depends only on geometry, and on the $L^\infty(\Omega)$-bound and ellipticity constant of the given coefficient
  function $\mu$.
\end{lemma}

We have shown in Lemmata~\ref{lem:localize-interior},~\ref{lem:localize-N}
and~\ref{lem:localize-interface} that all localized functions $\eta v_j$ for $j=0,\dots,m$ are
H\"older continuous of (possibly different) degree $\alpha_j$ with an estimate against
$f \in W^{-1,q}_D(\Omega)$ which depends only on geometry, and on the $L^\infty(\Omega)$-bound
and ellipticity constant of the given coefficient function $\mu$. Thus, if we choose $\alpha$ to
be the
minimum of the $\alpha_j$, the claim of Theorem~\ref{t-result} follows and we are done.

% Moreover, the extension of $\eta_j v$ by $0$ to the whole of $\Omega$ preserves this  H\"older continuity, since $\supp(\eta_j)$ has a positive distance to $\Omega \setminus  V_{\x_j}$. Finally, one observes that the H\"older norms of the elements $\eta_j v$ can be estimated by the norm $\|f\|_{W^{-1,q}_D(\Omega)}$, since all operations, as localisation, transformation and reflections respect the mutual continuous dependence of the right hand sides.

\end{document}